\documentclass[11pt]{article}
\usepackage{geometry}
\usepackage{color}
\usepackage{float}
\usepackage{textcomp}
\usepackage{amsthm}
\usepackage{amsmath}
\usepackage{amssymb}%
\usepackage{multirow}
\usepackage{changes}

\usepackage{mathtools}
\usepackage{booktabs}    
\usepackage{subfigure} 
\usepackage{caption} 
\usepackage{multirow}
\usepackage{multicol}    
\usepackage{array}       
\usepackage{graphicx}
\graphicspath{{figuras/}}
\usepackage{empheq,amsmath}
\usepackage{hyperref}

\usepackage
[ 
lined, 
boxruled, 
commentsnumbered
]
{algorithm2e}
\DecMargin{0.1cm} 

\newtheorem{theorem}{Theorem}[section]
\newtheorem{lemma}[theorem]{Lemma}
\newtheorem{corollary}[theorem]{Corollary}

\DeclareMathOperator{\argmin}{argmin}

\DeclareMathOperator{\dom}{dom}

\newcommand{\R}{\mathbb{R}}

\newcommand{\la}{{\langle}}
\newcommand{\ra}{{\rangle}}

\newcommand{\bi}{\begin{itemize}}
\newcommand{\ei}{\end{itemize}}
\newcommand{\ba}{\begin{array}}
\newcommand{\ea}{\end{array}}

\def\beq{\begin{equation}}
\def\eeq{\end{equation}}
\def\ba{\begin{array}}
\def\ea{\end{array}}
\def\beann{\begin{eqnarray*}}
\def\eeann{\end{eqnarray*}}
\def\bea{\begin{eqnarray}}
\def\eea{\end{eqnarray}}

\def\QR{\hfill \Box}
\def\Def{\stackrel{\mathrm{def}}{=}}

\begin{document}

\title{\textbf{First and zeroth-order implementations\\ 
		of the regularized  
		Newton method with lazy approximated Hessians}}


\author{
	Nikita Doikov \thanks{\'{E}cole Polytechnique F\'{e}d\'{e}rale de Lausanne (EPFL),
		Machine Learning and Optimization Laboratory (MLO), Switzerland (nikita.doikov@epfl.ch). The work was supported by the Swiss State Secretariat for Education, Research and
		Innovation (SERI) under contract number 22.00133. }  \and 
	Geovani N. Grapiglia \thanks{Universit\'{e} catholique de Louvain (UCLouvain), 
		Institute of Information and Communication Technologies, Electronics and Applied Mathematics (ICTEAM/INMA), Belgium (geovani.grapiglia@uclouvain.be).} 
}

\date{ September 5, 2023 }

\maketitle

\begin{abstract}
In this work, we develop first-order (Hessian-free) 
and zero-order (derivative-free) implementations of the 
Cubically regularized Newton method
for solving general non-convex optimization problems.
For that, we employ finite difference approximations
of the derivatives. We use a special adaptive search procedure in our algorithms, 
which simultaneously fits both the regularization constant and the parameters of the finite difference
approximations.  It makes our schemes free from the need to know the actual Lipschitz constants.
Additionally, we equip our algorithms with the lazy Hessian update
that reuse a previously computed Hessian approximation
matrix for several iterations.
Specifically, we prove the global complexity bound of $\mathcal{O}( n^{1/2} \epsilon^{-3/2})$
function and gradient evaluations for our new Hessian-free method,
and a bound of $\mathcal{O}( n^{3/2} \epsilon^{-3/2} )$
function evaluations for the derivative-free method, where 
$n$ is the dimension of the problem and $\epsilon$ is the desired accuracy for the gradient norm.
These complexity bounds significantly improve the previously known ones
in terms of the joint dependence on $n$ and $\epsilon$, for the first-order and 
zeroth-order non-convex optimization.
\end{abstract}

\section{Introduction}

\paragraph{Motivation.}

The Newton Method is a powerful algorithm for solving
numerical optimization problems.
Employing the matrix of second derivatives (the Hessian of the objective),
the Newton Method is  able to efficiently tackle \textit{ill-conditioned problems},
which can be very difficult for solving by the first-order Gradient Methods.

While the Newton Method has been remaining popular for many decades due to
 its exceptional practical performance, the study of its \textit{global} complexity bounds
is relatively recent.
One of the most theoretically established versions of this method is
the Cubically Regularized Newton Method \cite{nesterov2006cubic},
that achieves a global complexity of the order $\mathcal{O}( \epsilon^{-3/2} ) $
for finding a second-order stationary point for non-convex objective
with Lipschitz continuous Hessian, where  $\epsilon > 0$ is the desired accuracy for the gradient norm.
The corresponding complexity of the Gradient Method \cite{nesterov2018lectures}
for non-convex functions with Lipschitz continuous gradient is $\mathcal{O}( \epsilon^{-2})$,
which is significantly worse.
Thus, the Cubic Newton Method (CNM) achieves a \textit{provable improvement} of the global complexity,
as compared to the first-order methods.

In the recent years, there were developed many efficient modifications of CNM,
including 
\textit{adaptive} and \textit{universal} methods
\cite{cartis2011adaptive1,cartis2011adaptive2,grapiglia2017regularized,grapiglia2019accelerated,doikov2021minimizing,doikov2022super}
that does not require to know the actual Lipschitz constant of the Hessian
and that can automatically adapt to the best problem class among the functions with H\"older continuous derivatives,
\textit{accelerated} second-order schemes
\cite{nesterov2008accelerating,monteiro2013accelerated,nesterov2018lectures,grapiglia2019accelerated,kovalev2022first,carmon2022optimal}
with even improved convergence rates for convex functions and matching the 
lower complexity bounds \cite{agarwal2018lower,nesterov2018lectures,arjevani2019oracle}.

Clearly, we pay a significant price for the better convergence rates of CNM which is:
computation of second derivatives and solving a more difficult subproblem in each step.
Note that for some of the most difficult modern applications, our available information 
about the objective function $f(\cdot)$
can be restricted to the black-box
\begin{center}
	\textit{First-order oracle:} $\;\; x \; \mapsto \; \{ f(x), \nabla f(x) \}$,
\end{center}
or even to
\begin{center}
	\textit{Zeroth-order oracle:} $\;\; x \;\; \mapsto \; \{ f(x) \}$,
\end{center}
without a direct access to the problem structure and any ability to compute the second derivatives $\nabla^2 f(x)$ exactly.
Thus, in this black-box scenarios, we are interested to use
optimization schemes which efficiently employ only the information we have an access to.

First-order implementations of CNM were proposed and analysed in \cite{cartis2012oracle} and \cite{grapiglia2022cubic}. 
In both of these works, the methods employ finite-difference Hessian approximations, and complexity bounds of $\mathcal{O}(n\epsilon^{-3/2})$ calls of the oracle were proved,
where $n$ is the dimension of the problem.
In \cite{cartis2012oracle}, a zeroth-order implementation of CNM was also proposed, for which the authors showed a complexity bound of $\mathcal{O}(n^{2}\epsilon^{-3/2})$ calls of the oracle. 
At each iteration, methods in \cite{cartis2012oracle} and \cite{grapiglia2022cubic} require the computation of one or more Hessian approximations. Recently, in \cite{doikov2023second}, a second-order variant of CNM with \textit{lazy Hessians} was proposed, in which the same Hessian matrix is reused during $m \geq 1$ 
consecutive iterations (as in \cite{shamanskii1967modification}). Remarkably, the method with lazy Hessians retains the iteration complexity bound of $\mathcal{O}(\epsilon^{-3/2})$ for nonconvex problems. Moreover, when $m=n$, it requires in the worst-case a number of Hessian evaluations smaller by a factor of $\sqrt{n}$ in comparison with the standard CNM.

In this paper, we efficiently combine the use of finite-differences with the reuse of previously computed Hessian approximations to obtain new first and zeroth-order implementations of the CNM. Specifically, our algorithms employ adaptive searches by which the regularization parameters in the models and the finite-difference intervals are simultaneously adjusted (as in \cite{grapiglia2022cubic}).  Additionally, to improve the total oracle complexity of our schemes, we employ the lazy Hessian updates \cite{doikov2023second}, reusing each Hessian approximation for several consecutive steps. As the result, we obtain purely first-order (Hessian-free) and zeroth-order (derivative-free) implementations
of CNM that are adaptive and need, respectively, at most
$\mathcal{O}(n^{1/2}\epsilon^{-3/2})$ and  $\mathcal{O}(n^{3/2}\epsilon^{-3/2})$ 
calls of the oracle to find an $\epsilon$-approximate second-order stationary point of the objective function.
These complexity bounds significantly improve the corresponding bounds in \cite{cartis2012oracle} and \cite{grapiglia2022cubic} in terms of the dependence on $n$.
Note that our new methods also support \textit{composite problem} formulation (as, e.g. in \cite{grapiglia2019accelerated}), which include both unconstrained minimization and minimization with respect to simple convex constraints or additive regularization. In its turn, the smooth (and the difficult) part of the problem
can be non-convex.
Finally, we report the result of preliminary numerical experiments that illustrate the practical efficiency of the proposed methods.

\paragraph{Contents.}
In Section~\ref{SectionCubic} we introduce the inexact step of CNM,
which is the main primitive of all our algorithmic schemes. 
Section~\ref{SectionFiniteDiff} is devoted to the {finite difference} approximations 
of the second- and first-order derivatives of a smooth functions.
In Section~\ref{SectionHF}, we present first-order (\textit{Hessian-free}) implementation of CNM
and establish its global complexity bounds.
Section~\ref{SectionZO} contains zeroth-order (\textit{derivative-free}) implementation of CNM.
In Section~\ref{SectionLocal}, we establish local superlinear convergence for our schemes.
Section~\ref{SectionExperiments} presents illustrative numerical experiments. In Section~\ref{SectionDiscussion},
we discuss our results.

\paragraph{Notation and Assumptions.}
By $\| \cdot \|$ we denote
the standard Euclidean norm for vectors and the spectral norm for matrices,
while notation $\| \cdot \|_F$ is reserved for the matrix Frobenius norm.
We denote by $e_1, \ldots, e_n$ the standard basis vectors in $\R^n$.

We want to solve the following minimization problem
\beq \label{MainProblem}
\ba{rcl}
\min\limits_{x \in Q} \Bigl\{ 
F(x) & \Def &
f(x) + \psi(x) \Bigr\},
\ea
\eeq
where $Q \Def \dom \psi \subseteq \R^n$.
Function
$f : \R^n \to \R$ is a twice continuously differentiable, potentially
\textit{non-convex},
while the composite part $\psi: \R^n \to \R \cup \{ +\infty \}$ is a \textit{simple}
proper, closed, and convex, but
possibly non-differentiable (e.g. indicator of a given closed convex set $Q$).

Therefore, our goal  is to find a point $\bar{x} \in Q$ with a small (sub)gradient norm:
\beq \label{InexactSolution}
\ba{rcl}
\| \nabla f(\bar{x}) + \psi'(\bar{x}) \| & \leq & \epsilon,
\ea
\eeq
where $\psi'(\bar{x}) \in \partial \psi(\bar{x})$
and $\epsilon > 0$ is a desired tolerance.
We are aiming to find a point satisfying \eqref{InexactSolution},
using only \textit{first-order} or \textit{zeroth-order}   black-box oracle calls for $f$.
At the same time, the composite component $\psi$ is assumed
to be simple enough, such that the corresponding auxiliary minimization problems
that involve $\psi$ can be efficiently solved
(we present the form of the subproblem that we require to solve explicitly in the next section).

We assume that $F$ is bounded from below on $Q$ and denote 
$$
\ba{rcl}
F^{\star}  & \Def & 
\inf\limits_{x \in Q} F(x)
\;\; > \;\; -\infty.
\ea
$$
To characterize the smoothness of the differentiable part of the objective,
we assume the following:

\textbf{A1} The Hessian of $f$ is Lipschitz continuous, i.e.,
\beq \label{LipHess}
\ba{rcl}
\|\nabla^{2}f(y)-\nabla^{2}f(x)\|
& \leq & L\|y-x\|,\qquad \forall x,y \in \R^n,
\ea
\eeq
where $L \geq 0$ is the Lipschitz constant. Note that in all our methods, we do not need to know the exact value of $L$, estimating it 
\textit{automatically} with an adaptive procedure.

\section{Inexact Cubic Newton Step}
\label{SectionCubic}

In this section, we analyze one step of the Cubically regularized Newton method
with an \textit{approximate} second-order and first-order information. We also
assume that the step of the method is computed \textit{inexactly},
which would allow to apply our methods in the large scale setting.

Given $x \in Q$ and $\sigma>0$, let us define the models for $f(y)$ around $x$,
exact second-order model with cubic regularization:
\begin{equation}
\ba{rcl}
\Omega_{x,\sigma}(y)
& \Def & 
f(x) + \langle \nabla f(x),y-x \rangle
+ \dfrac{1}{2}\langle\nabla^{2}f(x)(y-x),y-x\rangle
+ \dfrac{\sigma}{6}\|y-x\|^{3},
\label{ExactModel}
\ea
\end{equation}
and an \textit{approximate model}:
\begin{equation}
\ba{rcl}
M_{x,\sigma}(y) & \Def & 
f(x)+\langle g,y-x\rangle+\dfrac{1}{2}\langle B(y-x),y-x\rangle+\dfrac{\sigma}{6}\|y-x\|^{3},
\ea
\label{ApproxModel}
\end{equation}
where $g\in\mathbb{R}^{n}$ is an approximation to $\nabla f(x)$ and $B\in\mathbb{R}^{n\times n}$ is an approximation to $\nabla^{2}f(z)$,
with some previous point $z \in \R^n$ from the past. In the simplest case, we can set $z := x$.
However, to reduce the iteration cost of our methods, we will use the same anchor point $z$ for several iterations
(that we call \textit{lazy Hessian updates}).

Note that due to the cubic regularizer, 
we can minimize model \eqref{ApproxModel} globally even when the quadratic part is non-convex.
Efficient techniques for solving such subproblems by using Linear Algebra tools or
gradient-based solvers
were extensively developed in the context of trust-region methods \cite{conn2000trust}
and for the Cubically regularized Newton methods
\cite{nesterov2006cubic,cartis2011adaptive1,cartis2011adaptive2,carmon2019gradient}.

Let us consider a minimizer 
for our approximate  model \eqref{ApproxModel} augmented  by the composite component:
\beq \label{CompositeSubproblem}
\boxed{
\ba{rcl}
x^+ & \approx & \argmin\limits_{y \in Q} 
\Bigl\{ M_{x, \sigma}(y) + \psi(y) \Bigr\}
\ea
}
\eeq
We will use such point $x^+$ as the main iteration step in all our methods.

Note that if $x^+$ is an \textit{exact solution} to \eqref{CompositeSubproblem},  then
the following first-order optimality condition holds (see, e.g. Theorem 3.1.23 in \cite{nesterov2018lectures}):
\beq \label{ExactStatCond}
\ba{rcl}
\la g + B(x^+ - x) + \frac{\sigma}{2}\|x^+ - x\| (x^+ - x), y - x^+ \ra
+ \psi(y) & \geq & \psi(x^+),
\quad \forall y \in Q.
\ea
\eeq
Hence, we have an explicit expression for a specific subgradient of $\psi$ at new point:
$$
\ba{rcl}
- g - B(x^+ - x) - \frac{\sigma}{2}\|x^+ - x\| (x^+ - x)
& \overset{\eqref{ExactStatCond}}{\in} & \partial \psi(x^+).
\ea
$$
Thus, usually for any solver of \eqref{CompositeSubproblem},
along with $x^+$ we are able to compute the corresponding subgradient vector
as well.

In what follows, we will consider \textit{inexact minimizers} of our model. 
First, we provide the bound for the new gradient norm.

\begin{lemma} \label{LemmaNewGrad}
	Let $x^{+}$ be an inexact minimizer of subproblem \eqref{CompositeSubproblem} satisfying the following condition, for some $\theta \geq 0$:
	\beq \label{InexactCond}
	\ba{rcl}
	\| \nabla M_{x, \sigma}(x^{+}) + \psi'(x^{+}) \| & \leq & \theta \|x^{+} - x\|^2,
	\ea
	\eeq
	for a certain $\psi'(x^{+})  \in \partial \psi(x^{+})$.
	Let, for some $\delta_g, \delta_B \geq 0$, it hold that
	\beq \label{GradHessApprox}
	\ba{rcl}
	\|g - \nabla f(x) \| & \leq & \delta_g, \\
	\\
	\| B - \nabla^2 f(z) \| & \leq & \delta_B.
	\ea
	\eeq
	Then, we have
	\beq \label{NewGradBound}
	\ba{rcl}
	\| \nabla f(x^{+}) + \psi'(x^{+})  \| & \leq & 
	\bigl(  \theta + \frac{\sigma + L}{2} \bigr) r^2
	\, + \, \bigl( \delta_B +  L \|x - z\| \bigr) r \, + \,  \delta_g,
	\ea
	\eeq
	where $r := \|x^+ - x\|$.
\end{lemma}
\proof
Indeed, 
$$
\ba{rcl}
\| \nabla f(x^{+}) + \psi'(x^{+})  \|
& \leq & 
\| \nabla f(x^{+}) - \nabla \Omega_{x, \sigma}(x^{+}) \|
+ \| \nabla \Omega_{x, \sigma}(x^+)
- \nabla M_{x, \delta}(x^+) \| \\
\\
& & \quad
+ \; \| \nabla M_{x, \sigma}(x^+) + \psi'(x^+) \| \\
\\
& = &
\bigl\| \nabla f(x^+) - \nabla f(x) - \nabla^2 f(x)(x^+ - x) - \frac{\sigma}{2}r(x^+ - x)  \bigr\| \\
\\
& & \quad
+ \; \| \nabla f(x) - g + (\nabla^2 f(x) - B)(x^+ - x) \| + \| \nabla M_{x, \sigma}(x^+) + \psi'(x^{+})  \| \\
\\
& \overset{\eqref{LipHess}, \eqref{InexactCond}}{\leq} &
\bigl( \theta +  \frac{\sigma + L}{2} \bigr) r^2 
\, + \, \| \nabla^2 f(x) - B \| r 
\, + \, \| \nabla f(x) - g\| \\
\\
& \overset{\eqref{LipHess}, \eqref{GradHessApprox}}{\leq} &
\bigl( \theta +  \frac{\sigma + L}{2} \bigr) r^2 
\, + \, \bigl(  \delta_B + L\|x - z\| \bigr) r
\, + \, \delta_g. \QR
\ea
$$

Now, we can express the progress of one step in terms of the objective function value.

\begin{lemma} \label{LemmaNewFunc}
	Let $x^+$ satisfy the following condition: 
	\beq \label{XplusCond}
	\ba{rcl}
	M_{x, \sigma}(x^{+}) + \psi(x^+) & \leq & F(x), 
	\ea
	\eeq
	and let $g$ and $B$ satisfy \eqref{GradHessApprox} for some $\delta_g, \delta_B \geq 0$.
	Then, we have
	\beq \label{NewFuncProgress}
	\ba{rcl}
	F(x) - F(x^+) & \geq & 
	\frac{\sigma - L}{6}r^3
	- \frac{1}{2}(\delta_B + L\|x - z\|) r^2
	- \delta_g r,
	\ea
	\eeq
	where $r := \|x^+ - x\|$.
\end{lemma}
\proof
Indeed, we have
$$
\ba{rcl}
F(x^+) & \overset{\eqref{LipHess}}{\leq} &
\Omega_{x, L}(x^+) + \psi(x^+) \\
\\
& = & 
f(x) + \la \nabla f(x), x^{+} - x \ra
+ \frac{1}{2} \la \nabla^2 f(x)(x^+ - x), x^+ - x \ra
+ \frac{L}{6}\|x^+ - x\|^3 + \psi(x^+) \\
\\
& = & 
M_{x, \sigma}(x^+)
+ \la \nabla f(x) - g, x^{+} - x \ra
+ \frac{1}{2} \la (\nabla^2 f(x) - B)(x^+ - x), x^+ - x \ra \\
\\
& & \quad
\; + \; \frac{L - \sigma}{6}\|x^+ - x\|^3 + \psi(x^+) \\
\\
& \overset{\eqref{XplusCond}, \eqref{GradHessApprox}, \eqref{LipHess}}{\leq} &
F(x)
+ \delta_g r
+ \frac{1}{2} (\delta_B + L\|x - z\|) r^2
+ \frac{L - \sigma}{6} r^3, 
\ea
$$
and this is \eqref{NewFuncProgress}.
\qed

Finally, we analyze the smallest eigenvalues for the Hessian of our problem. Let us 
consider the case when
the composite part $\psi$ is \textit{twice differentiable}, so the Hessian
of the full objective in \eqref{MainProblem} is well-defined. 
Then, we denote
\beq \label{XiDef}
\ba{rcl}
\xi(y) & \Def & 
\max\Bigl\{  -\lambda_{\min}( \nabla^2 F(y) ), 0  \Bigr\}, \qquad y \in Q.
\ea
\eeq
Thus, the value of $\xi(y) \geq 0$ indicates how big is the negative part of the smallest eigenvalue of the Hessian
at point $y$.

Note that if $x^{+}$ is an \textit{exact solution} to our subproblem \eqref{CompositeSubproblem},
we can use the following second-order optimality condition (see, e.g. Theorem 1.2.2 in \cite{nesterov2018lectures}):
\beq \label{SOStat}
\ba{rcl}
B + \frac{\sigma}{2} \|x^+ - x\| I + \frac{\sigma}{2r} (x^+ - x) (x^+ - x)^{\top} + \nabla^2 \psi(x^+) & \succeq & 0,
\ea
\eeq
where $I$ is identity matrix. In order to provide the guarantee for $\xi(x^+)$,
we can use the relaxed version of \eqref{SOStat}.

\begin{lemma} \label{LemmaNewHess}
	Let $\psi$ be twice differentiable. Let $x^+$ satisfy the following
	condition, for some $\theta \geq 0$:
	\beq \label{HessInexactCondition}
	\ba{rcl}
	B + \theta \|x^+ - x\| I + \nabla^2 \psi(x^+) & \succeq & 0.
	\ea
	\eeq
	Let, for some $\delta_B \geq 0$, it hold that
	\beq \label{SOInexHess}
	\ba{rcl}
	\| B - \nabla^2 f(z) \| & \leq & \delta_B.
	\ea
	\eeq 
	Then, we have
	\beq \label{XiBound}
	\ba{rcl}
	\xi(x^+) & \leq & 
	(L + \theta) r + L\|x - z\| + \delta_B.
	\ea
	\eeq
	where $r := \|x^+ - x\|$.
\end{lemma}
\proof
Using Lipschitzness of the Hessian of $f$ \eqref{LipHess}, we have
$$
\ba{rcl}
\nabla^2 F(x^+) & \succeq & 
\nabla^2 f(x) + \nabla^2 \psi(x^+) - Lr I \\
\\
& \succeq &
\nabla^2 f(z) + \nabla^2 \psi(x^+) - (Lr + L\|x - z\|) I \\
\\
& \overset{\eqref{SOInexHess}}{\succeq} &
B + \nabla^2 \psi(x^+)  - (Lr + L\|x - z\| - \delta_B) I \\
\\
& \overset{\eqref{HessInexactCondition}}{\succeq} &
- ( Lr + \theta r + L \|x - z\| + \delta_B ) I,
\ea
$$ 
which leads to \eqref{XiBound}.
\qed

Let us combine all our lemmas together. We justify the following bound for the progress of one step
for our inexact composite Cubic Newton Method (CNM):

\begin{theorem} \label{ThStep}
	Let $\sigma \geq 2L$. 
	Let $x^+$ be an inexact minimizer of model \eqref{ApproxModel} satisfying the following two conditions,
	for a certain $\psi'(x^+) \in \partial \psi(x^+)$:
	\beq \label{InexactThStep}
	\ba{rcl}
	\| \nabla M_{x, \sigma}(x^+) + \psi'(x^+) \| & \leq & \frac{\sigma}{4}\|x^+ - x\|^2, \\
	\\
	M_{x, \sigma}(x^+) + \psi(x^+) & \leq & F(x),
	\ea
	\eeq
	where $g$ and $B$  satisfy \eqref{GradHessApprox} for some $\delta_g, \delta_B \geq 0$. Then, we have
	\beq \label{InexactProgressThStep}
	\ba{rcl}
	F(x) - F(x^+) 
	& \geq &
	\frac{1}{3 \cdot 2^6 \sigma^{1/2}} \| \nabla f(x^+) + \psi'(x^+) \|^{3/2} +
	\mathcal{E},
	\ea
	\eeq 
	where 
	$$
	\ba{rcl}
	\mathcal{E} & \Def & \frac{\sigma}{48} \|x^+ - x\|^3
	- \frac{171}{\sigma^2}\Bigl[ \delta_B^3 + L^3 \|x - z\|^3  \Bigr]
	- \frac{3}{\sigma^{1/2}} \delta_g^{3/2}.
	\ea
	$$
	Assume additionally that $\psi$ is twice differentiable, and $x^+$ satisfies the following extra condition:
	\beq \label{InexatSOThStep}
	\ba{rcl}
	B + \sigma \|x^+ - x\| I + \nabla^2 \psi(x^+) & \succeq & 0.
	\ea
	\eeq
	Then, we can improve \eqref{InexactProgressThStep}, as follows:
	\beq \label{InexactProgressSOThStep}
	\ba{rcl}
		F(x) - F(x^+) 
	& \geq &
	\max\Bigl\{
	\frac{1}{3 \cdot 2^6 \sigma^{1/2}} \| \nabla f(x^+) + \psi'(x^+) \|^{3/2},
	\;
	\frac{1}{2 \cdot 3^6 \sigma^2} \bigl[ \xi(x^+) \bigr]^3
	\Bigr\}
	+ \mathcal{E}.
	\ea
	\eeq
	
\end{theorem}
\proof We denote $r := \|x^+ - x\|$.
Firstly, we bound the negative terms from \eqref{NewFuncProgress}, by using Young's inequality:
$ab \leq \frac{a^3}{3} + \frac{2b^{3/2}}{3}$, $a, b \geq 0$. We have
\beq \label{FuncProgThB1}
\ba{rcl}
\frac{1}{2} (\delta_B + L \|x - z\|) r^2 
& = &
\Bigl[ \frac{\sigma^{2/3} r^2}{2^{10/3}}  \Bigr] 
\cdot \Bigl[ 
\frac{2^{10 / 3}}{2 \sigma^{2/3}} \cdot \bigl(  \delta_B + L\|x - z\|  \bigr) \Bigr] \\
\\
& \leq & 
\frac{2}{3} \Bigl[  \frac{\sigma^{2/3} r^2}{2^{10/3}}  \Bigr]^{3/2}
+ \frac{1}{3}
\Bigl[ 
\frac{2^{10 / 3}}{2 \sigma^{2/3}} \cdot \bigl(  \delta_B + L \|x - z\|  \bigr) \Bigr]^3 \\
\\
& = &
\frac{\sigma r^3}{3 \cdot 2^4}
+ \frac{2^7}{3\sigma^2} \bigl(  \delta_B + L \|x - z\| \bigr)^3 
\;\; \leq \;\;
\frac{\sigma r^3}{3 \cdot 2^4}
+ \frac{2^9}{3 \sigma^2} \bigl(  \delta_B^3 + L^3\|x - z\|^3 \bigr),
\ea
\eeq
and
\beq \label{FuncProgThB2}
\ba{rcl}
\delta_g r & = & 
\Bigl[  \frac{\sigma^{1/3} r}{2^{4/3}} \Bigr] \cdot 
\Bigl[ \frac{2^{4/3} \delta_g}{\sigma^{1/3}} \Bigr]
\;\; \leq \;\;
\frac{\sigma r^3}{48} + \frac{2^3 \delta_g^{3/2}}{3\sigma^{1/2}}.
\ea
\eeq
Therefore, for the functional progress, we obtain
\beq \label{FuncProgThStep1}
\ba{rcl}
F(x) - F(x^+) & \overset{ \eqref{NewFuncProgress}, \eqref{FuncProgThB1}, \eqref{FuncProgThB2} }{\geq} & 
\frac{\sigma}{24} r^3
\; - \; \frac{2^9}{3 \sigma^2} \bigl( \delta_B^3 + L^3\|x - z\|^3  \bigr)
\; - \; \frac{2^3 \delta_g^{3/2}}{3\sigma^{1/2}}.
\ea
\eeq
Secondly, we can relate $r$ and the new gradient norm by using \eqref{NewGradBound}. We get
\beq \label{FuncProgThStep2}
\ba{rcl}
\| \nabla f(x^+) + \psi'(x^+) \|^{3/2} & \overset{\eqref{NewGradBound}}{\leq} &
\Bigl(  
\sigma r^2
+ \delta_B r + L \|x - z\| r + \delta_g 
\Bigr)^{3/2} \\
\\
& \overset{(*)}{\leq} & 
2 \sigma^{1/2}
\Bigl(  \sigma r^3 
\; + \; \frac{\delta_B^{3/2} r^{3/2}}{\sigma^{1/2}} 
\; + \; \frac{L^{3/2} \|x - z\|^{3/2} r^{3/2}}{\sigma^{1/2}} 
\; + \;  \frac{ \delta_g^{3/2}}{\sigma^{1/2}} 
\Bigr)\\
\\
& \overset{(**)}{\leq} & 
2 \sigma^{1/2}
\Bigl( 
2 \sigma r^3
+ \frac{\delta_B^{3}}{2 \sigma^2}
+ \frac{L^3 \|x - z\|^3}{2 \sigma^2} + \frac{\delta_g^{3/2}}{\sigma^{1/2}},
\Bigr)
\ea
\eeq
where we used in $(*)$ Jensen's inequality: $(\sum_{i = 1}^4 a_i )^{3/2} \leq 2 \sum_{i = 1}^4 a_i^{3/2}$
for non-negative numbers $\{ a_i \}_{i = 1}^4$, and 
in $(**)$ Young's inequality: $ab \leq \frac{a^2}{2} + \frac{b^2}{2}$, $a, b \geq 0$.
Rearranging the terms, we obtain
\beq \label{FuncProgThStep3}
\ba{rcl}
\sigma r^3 & \overset{\eqref{FuncProgThStep2}}{\geq} &
\frac{1}{4 \sigma^{1/2}} \| \nabla f(x^+) + \psi'(x^+) \|^{3/2}
- \frac{1}{4\sigma^2} \bigl( \delta_B^3 + L^3\|x - z\|^3 \bigr)
- \frac{\delta_g^{3/2}}{2 \sigma^{1/2}}.
\ea
\eeq
Combining \eqref{FuncProgThStep1} and \eqref{FuncProgThStep3} gives \eqref{InexactProgressThStep}.

Finally, assuming twice differentiability of the composite part and using Lemma~\ref{LemmaNewHess}
for the extra condition \eqref{InexatSOThStep} on $x^+$, we get
\beq \label{Xi3Bound}
\ba{rcl}
\bigl[ \xi(x^+) \bigr]^3
& \overset{\eqref{XiBound}}{\leq} &
\Bigl[   \frac{3}{2} \sigma r + L\|x - z\| + \delta_B  \Bigr]^3 \\
\\
& \overset{(*)}{\leq} &
\frac{3^5}{2^3} \sigma^3 r^3
+ 3^2 L^3 \|x - z\|^3 
+ 3^2 \delta_B^3,
\ea
\eeq
where we used in $(*)$ Jensen's inequality: $( \sum_{i = 1}^3 a_i )^3 \leq 3^2 \sum_{i = 1}^3 a_i$
for non-negative numbers $\{ a_i \}_{i = 1}^3$.
Hence, rearranging the terms, we obtain
$$
\ba{rcl}
\sigma r^3
& \overset{\eqref{Xi3Bound}}{\geq} &
\frac{2^3}{3^5 \sigma^2} \bigl[  \xi(x^+) \bigr]^3
- \bigl(  \frac{2}{3} \bigr)^3
\frac{1}{\sigma^2} \bigl( \delta_B^^3 + L^3\|x - z\|^3  \bigr).
\ea
$$
Combining it with \eqref{FuncProgThStep1}
justifies the improved bound \eqref{InexactProgressSOThStep}. 
\qed

\section{Finite Difference Approximations}
\label{SectionFiniteDiff}

In this section, we recall important bounds on finite difference
approximations for the Hessian and for the gradient of our objective.

Let us start with the first-order approximation of the Hessian, that will lead
us to the first-order (Hessian-free) implementation of the Cubic Newton Method.
See, e.g., Lemma~3 in \cite{grapiglia2022cubic}.

\begin{lemma} \label{LemmaHessFO}
	Suppose that A1 holds.
	Given $\bar{x}\in\mathbb{R}^{n}$ and $h>0$, let $A\in\mathbb{R}^{n\times n}$ be defined by
	\begin{equation}
	\label{HessFODefA}
	\ba{rcl}
	A  &= & 
	\left[\dfrac{\nabla f(\bar{x}+he_{1})-\nabla f(\bar{x})}{h},\ldots,\dfrac{\nabla f(\bar{x}+he_{n})-\nabla f(\bar{x})}{h}\right].
	\ea
	\end{equation}
	Then, the matrix
	\begin{equation} \label{HessFODefB}
	\ba{rcl}
	B &= & \frac{1}{2}\left(A+A^{\top}\right)
	\ea
	\end{equation}
	satisfies
	\begin{equation} \label{HessFOBound}
	\ba{rcl}
	\|B-\nabla^{2}f(\bar{x})\| & \leq & \frac{\sqrt{n} L}{2} h.
	\ea
	\end{equation}
\end{lemma}

Now, let us consider zeroth-order approximations of the derivatives,
that requires computing only the objective function value (see, e.g.,  Section~7.1 in \cite{nocedal2006numerical}).
We establish explicit bounds necessary for the analysis of our methods
and provide their proofs to ensure completeness of our presentation.
The following lemma gives a zeroth-order approximation guarantee for the gradient.

\begin{lemma} \label{LemmaGradZO}
	Suppose that A1 holds. Given $\bar{x}\in\mathbb{R}^{n}$ and $h>0$, let $g\in\mathbb{R}^{n}$ be defined by
	\begin{equation} \label{GradZODef}
	\ba{rcl}
	g_{i} & = & \dfrac{f(\bar{x}+he_{i})-f(\bar{x}-he_{i})}{2h},\quad i=1,\ldots,n.
	\ea
	\end{equation}
	Then,
	\beq \label{GradZOBound}
	\ba{rcl}
	\| g - \nabla f(\bar{x}) \| & \leq & \frac{\sqrt{n} L}{6} h^2.
	\ea
	\eeq
\end{lemma}

\begin{proof}
	By A1 we have
	\begin{equation}
	\ba{rcl}
	\left|f(\bar{x}+he_{i})-f(\bar{x})-h\langle\nabla f(\bar{x}),e_{i}\rangle-\dfrac{h^{2}}{2}\langle\nabla^{2}f(\bar{x})e_{i},e_{i}\rangle\right|
	& \leq & \frac{Lh^{3}}{6}
	\ea
	\label{eq:2.22}
	\end{equation}
	and
	\begin{equation}
	\ba{rcl}
	\left|f(\bar{x})-h\langle\nabla f(\bar{x}),e_{i}\rangle+\dfrac{h^{2}}{2}\langle\nabla^{2}(\bar{x})e_{i},e_{i}\rangle-f(\bar{x}-he_{i})\right|
	& \leq &
	\frac{Lh^{3}}{6}.
	\ea
	\label{eq:2.23}
	\end{equation}
	Summing (\ref{eq:2.22}) and (\ref{eq:2.23}) and using the triangle inequality, we get
	\beq \label{GFOinterm}
	\ba{rcl}
	\left|f(\bar{x}+he_{i})-f(\bar{x}-he_{i})-2h\left[\nabla f(\bar{x})\right]_{i}\right|
	& \leq & \frac{Lh^{3}}{3}
	\ea
	\eeq
	Therefore,
	$$
	\ba{rcl}
	|g_i - [\nabla f(\bar{x})]_i |& = &
	\left| 
	\frac{f(\bar{x} + he_i) - f(\bar{x} - he_i)}{2h}
	- [\nabla f(\bar{x})]_i
	\right|
	\;\; \overset{\eqref{GFOinterm}}{\leq} \;\;
	\frac{L h^2}{6}.
	\ea
	$$
	Thus, we conclude
	$$
	\ba{rcl}
	\|g - \nabla f(\bar{x})\|
	& \leq &
	\sqrt{n}\|g-\nabla f(\bar{x})\|_{\infty}
	\;\; \leq \;\; \frac{\sqrt{n}L}{6}h^{2}.
	\ea
	$$
\end{proof}

Finally, we provide a zeroth-order approximation guarantee for the Hessian.

\begin{lemma} \label{LemmaHessZO}
	Suppose that A1 holds. Given $\bar{x}\in\mathbb{R}^{n}$ and $h>0$, let $A\in\mathbb{R}^{n\times n}$ be defined by
	\begin{equation}  \label{HessZO}
	\ba{rcl}
	A_{ij}
	& = & \dfrac{f(\bar{x}+he_{i}+he_{j})-f(\bar{x}+he_{i})-f(\bar{x}+he_{j})-f(\bar{x})}{h^{2}},\quad i,j=1,\ldots,n.
	\ea
	\end{equation}
	Then, the matrix
	\begin{equation} \label{HessZOBDef}
	\ba{rcl}
	B & = & \frac{1}{2}\left(A+A^{\top}\right)
	\ea
	\end{equation}
	satisfies
	\begin{equation} \label{HessZOBBound}
	\ba{rcl}
	\|B-\nabla^{2}f(\bar{x})\|
	& \leq &  \frac{2 nL}{3}h.
	\ea
	\end{equation}
\end{lemma}

\begin{proof}
	By A1 we have the following inequalities:
	\begin{equation} \label{HessZOEqBound1}
	\ba{cl}
	& \Bigl|f(\bar{x} + h e_{i} + h e_{j}) - f(\bar{x}) - 
	h\langle\nabla f(\bar{x}),e_{i}\rangle
	- h\langle\nabla f(\bar{x}),e_{j}\rangle \\
	\\
	& \;
	- \frac{h^{2}}{2}\langle\nabla^{2}f(\bar{x})e_{i},e_{i}\rangle-h^{2}\langle\nabla^{2}f(\bar{x})e_{i},e_{j}\rangle
	- \frac{h^{2}}{2}\langle\nabla^{2}f(\bar{x})e_{j},e_{j}\rangle
	\Bigr|
	\;\; \leq \;\; \frac{Lh^{3}}{3},
	\ea
	\end{equation}
	\begin{equation} 	\label{HessZOEqBound2}
	\ba{c}
	\Bigl|f(\bar{x}) + h\langle\nabla f(\bar{x}),e_{i}\rangle + \frac{h^{2}}{2}\langle\nabla^{2}f(\bar{x})e_{i},e_{i}\rangle
	- f(\bar{x}+he_{i}) \Bigr|
	\;\; \leq \;\; \frac{Lh^{3}}{6},
	\ea
	\end{equation}
	and
	\begin{equation} \label{HessZOEqBound3}
	\ba{c}
	\Bigl|f(\bar{x}) + h\langle\nabla f(\bar{x}),e_{j}\rangle 
	+ \frac{h^{2}}{2}\langle\nabla^{2}f(\bar{x})e_{j},e_{j}\rangle
	- f(\bar{x}+he_{j}) \Bigr|
	\;\; \leq \;\; \frac{Lh^{3}}{6}
	\ea
	\end{equation}
	Summing \eqref{HessZOEqBound1}-\eqref{HessZOEqBound2}, and using the triangle inequality, we get
	\begin{equation*}
	\ba{c}
	\Bigl|f(\bar{x}+he_{i}+he_{j})-f(\bar{x}+he_{i})-f(\bar{x}+he_{j})
	 + f(\bar{x})-h^{2} \langle\nabla^{2}f(\bar{x})e_{i},e_{j}\rangle\Bigr|\;\; \leq \;\;\frac{2Lh^{3}}{3}
	 \ea
	\end{equation*}
	Hence,
	\begin{equation*}
	\ba{c} h^{2}\Bigl|\frac{f(\bar{x}+he_{i}+he_{j})-f(\bar{x}+he_{i})-f(\bar{x}+he_{j})+f(\bar{x})}{h^{2}}-\left[\nabla^{2}f(\bar{x})\right]_{ij}\Bigr| \;\; \leq \;\; \frac{2Lh^{3}}{3}
	\ea
	\end{equation*}
	and, consequently,
	\begin{equation*}
	\ba{rcl}\left|A_{ij}-\left[\nabla^{2}f(\bar{x})\right]_{ij}\right|
	& \leq &
	\frac{2L}{3}h.
	\ea
	\end{equation*}
	Thus, we finally obtain
	\begin{equation*}
	\ba{rcl}
	\|B-\nabla^{2}f(\bar{x})\|\leq \|A-\nabla^{2}f(\bar{x})\|
	& \leq & n\|A-\nabla^{2}f(\bar{x})\|_{\max}\leq\frac{2nL}{3}h,
	\ea
	\end{equation*}
	which is the required bound.
\end{proof}

\section{Hessian-Free CNM with Lazy Hessians}
\label{SectionHF}

Let us present our first algorithm, which is the \textit{Hessian-free} implementation
of the Cubic Newton Method (CNM) \cite{nesterov2006cubic}. 
In each iteration of our algorithm, we use an adaptive search to fit simultaneously
the regularization constant $\sigma$
and the parameter $h$ of finite difference approximation of the Hessian
(see Lemma~\ref{LemmaHessFO}). Therefore, our algorithm does not need to fix 
these parameters in advance, adjusting them automatically.

After the new approximation $B_{k, \ell} \approx \nabla^2 f(x_k)$ of the Hessian is computed, 
where $k \geq 0$ is the current iteration and $\ell$ is the adaptive search index,
we keep using the same matrix $B_{k, \ell}$ for the next $m$ Cubic Newton steps~\eqref{CompositeSubproblem},
where $m \geq 1$ is our global key parameter.

If we set $m := 1$, it means we update the Hessian approximation each Cubic Newton step,
which can be costly from the computational point of view. Instead, we can
use $m > 1$ (lazy Hessian updates~\cite{doikov2023second}),
that reuses the same Hessian approximation for several steps
and thus reduces the arithmetic complexity.

Let us denote by $(\hat{x}, \alpha) = \text{\ttfamily CubicSteps}(x, B, \sigma, m, \epsilon)$
an auxiliary procedure that performs $m$ inexact Cubic Newton steps
\eqref{CompositeSubproblem},
starting from point $x \in Q$ and
using the same given matrix $B = B^{\top}$ and regularization constant $\sigma > 0$
for all steps, while recomputing the gradients. Parameter $\epsilon > 0$
is used for validating a certain stopping condition. 
We can write this procedure in the algorithmic form, as follows.

\begin{algorithm}[h!]
	\caption{$\text{\ttfamily CubicSteps}(x, B, \sigma, m, \epsilon)$} \label{alg:HessianFree}
	\SetKwInOut{Input}{input}\SetKwInOut{Output}{output}
	\BlankLine
	\textbf{Step 0.} Set $x_0 := x$ and $t := 0$.
	\BlankLine
	{\bf Step 1.} 	If $t = m$ then stop and \textbf{return} $(x_t, \, \text{\ttfamily success})$.
	\BlankLine
	{\bf Step 2.}  Compute $x_{t + 1}$ as an approximate solution to the subproblem
	$$
	\ba{c}
	\min\limits_{y \in Q} \Bigl\{ \, M_{x_{t}, \sigma}(y) + \psi(y) \, \Bigr\},
	\qquad \text{where}
	\ea
	$$
	$$
	\ba{rcl}
	M_{x_{t}, \sigma}(y) 
	& \equiv &
	f(x_t) + \la \nabla f(x_t), y - x_t \ra
	+ \frac{1}{2}\la B(y - x_t), y - x_t \ra
	+ \frac{\sigma}{6}\|y - x_t\|^3
	\ea
	$$
	such that
	\beq \label{FOMCond}
	\ba{rcl}
	M_{x_{t},\sigma}(x_{t + 1}) + \psi(x_{t + 1}) &\leq & F(x_{t}) \quad \text{and} \\
	\\ 
	\|\nabla M_{x_{t}, \sigma} ( x_{t + 1} )   +  \psi'(x_{t + 1}) \| & \leq &
	\frac{\sigma}{4} \|x_{t + 1} - x_t  \|^{2} \quad \text{for some} \;\; \psi'(x_{t + 1}) \in \partial \psi(x_{t + 1}),
	\ea
	\eeq
	and (\textbf{optionally}, if $\psi$ is twice differentiable) such that
	\beq \label{FOMCond2}
	\ba{rcl}
	B + \sigma \|x_{t + 1} - x_t \| I + \nabla^2 \psi(x_{t + 1}) & \succeq & 0.
	\ea
	\eeq
	\BlankLine
	{\bf Step 3.}
	If $\| \nabla f(x_{t + 1}) + \psi'(x_{t + 1}) \| \leq \epsilon$ then stop and \textbf{return} $(x_{t + 1}, \, \text{\ttfamily solution})$.
	\BlankLine
	
	{\bf Step 4.} If $F(x_0) - F(x_{t + 1}) \geq \frac{\epsilon^{3/2}}{384 \sigma^{1/2}} (t + 1)$
	holds then set $t := t + 1$ and go to Step 1. Otherwise, stop and \textbf{return} $(x_{t + 1}, \text{\ttfamily halt})$.
	\BlankLine
\end{algorithm}

This procedure returns 
the resulting point $\hat{x} \in Q$ and a status variable 
$$
\ba{rcl}
\alpha &\in & 
\{\text{\ttfamily success}, \, \text{\ttfamily solution}, \, \text{\ttfamily halt} \}
\ea
$$
that corresponds respectfully to 
finishing all the steps \textit{successfully}, 
finding a point with \textit{small gradient norm}, 
and \textit{halting} the procedure due to insufficient progress in terms of the objective function.
In the last case, we will need to update our estimates $\sigma$ and $B$
adaptively
and restart this procedure with new parameters.

The next lemma shows that for sufficiently big value of $\sigma$ and 
 small enough $h$ (the parameter of finite difference approximation of the Hessian),
the result of Algorithm~\ref{alg:HessianFree} always belongs to $\{ \text{\ttfamily success}, \, \text{\ttfamily solution}   \}$,
that it either makes a significant progress in the function value, or solves the initial problem~\eqref{MainProblem}.

\begin{lemma} \label{LemmaHF}
	Suppose that A1 holds. Given $x \in Q$, 
	$\epsilon > 0$, $\sigma > 0$, and $m \in \mathbb{N} \setminus \{0 \}$,
	let $(\hat{x}, \alpha)$ be the corresponding output of Algorithm~\ref{alg:HessianFree}
	with $B = \frac{1}{2}(A + A^{\top})$, where
	\beq \label{HFLemmaA}
	\ba{rcl}
	A & = & 
	\Bigl[ 
	\frac{\nabla f(x + he_1) - \nabla f(x)}{h}, \ldots,
	\frac{\nabla f(x + h e_n) - \nabla f(x)}{h}
	\Bigr]
	\ea
	\eeq
	for some $h > 0$. If
	\beq \label{HFSimgaH}
	\ba{rcl}
	\sigma & \geq & 2^4	\bigl(  \frac{2}{3} \bigr)^{\frac{1}{3}} mL
	\qquad \text{and} \qquad
	h \;\; \leq \;\;
	\Bigl[ 
	\frac{3 \sigma^{3/2} \epsilon^{3/2}}{ 2^7 \cdot (192) n^{3/2} L^3 }
	\Bigr]^{\frac{1}{3}},
	\ea
	\eeq
	then either $\alpha = \text{\normalfont \ttfamily solution}$ (and thus $\| \nabla f(\hat{x}) + \psi'(\hat{x}) \| \leq \epsilon$), or
	 $\alpha = \text{\normalfont \ttfamily success}$ and so we have
	\beq \label{FOFuncProgress}
	\ba{rcl}
	F(x) - F(\hat{x}) & \geq & \frac{\epsilon^{3/2}}{2 \cdot (192) \sigma^{1/2}} m.
	\ea
	\eeq
\end{lemma}
\begin{proof}
Suppose that 
\beq \label{BigGrad}
\ba{rcl}
\| \nabla f(\hat{x}) + \psi'(\hat{x})  \| & > & \epsilon.
\ea
\eeq 
Hence, $\alpha \not= \text{\ttfamily solution}$.

Let us denote by $t^{\star}$ the last value of $t$ checked in Step 1.
Clearly, $t^{\star} \leq m$ and we need just to prove that $t^{\star} = m$.
Suppose that $t^{\star} < m$, and hence inequality in Step 4
of the algorithm does not hold for $t := t^{\star}$.

It follows from \eqref{HFLemmaA} and
Lemma \ref{LemmaHessFO} that
\beq \label{FOBBound}
\ba{rcl}
\| B - \nabla^2 f(x) \| & \leq & \delta_B
\ea
\eeq
for 
$$
\ba{rcl}
\delta_B & = & \frac{\sqrt{n} L}{2} h.
\ea
$$
Then, by the second inequality in \eqref{HFSimgaH} we get
\beq \label{DeltaBBound}
\ba{rcl}
\frac{2^9}{3 \sigma^2} \delta_B^3
& = &
\frac{2^9}{3 \sigma^2} \cdot \frac{n^{3/2} L^3}{2^3} \cdot h^3
\;\; \leq \;\;
\frac{\epsilon^{3/2}}{2 \cdot (192) \sigma^{1/2}}.
\ea
\eeq
Hence, 
in view of \eqref{FOMCond} and \eqref{FOBBound}, Theorem~\ref{ThStep} with $\delta_g := 0$ and $z := x = x_0$ gives
\beq \label{FuncProgLemma1}
\ba{rcl}
F(x_{t}) - F(x_{t + 1}) & \geq & 
\frac{\sigma}{48} \|x_{t + 1} - x_{t}\|^3
+ \frac{1}{192 \sigma^{1/2}} \| \nabla f(x_{t + 1}) + \psi'(x_{t + 1}) \|^{3/2}  \\
\\
& & 
\quad 
- \; \frac{2^9}{3 \sigma^2} \delta_B^3
- \frac{2^9 L^3}{3 \sigma^2} \|x_{t} - x_0\|^3 \\
\\
& \overset{\eqref{DeltaBBound}}{\geq} &
\frac{\sigma}{48} \|x_{t + 1} - x_{t}\|^3 
+ \frac{1}{192 \sigma^{1/2}} \| \nabla f(x_{t + 1}) + \psi'(x_{t + 1}) \|^{3/2}  \\
\\
& &
\quad
- \; \frac{\epsilon^{3/2}}{2(192)\sigma^{1/2}}
- \frac{2^9 L^3}{3 \sigma^2} \|x_{t} - x_0\|^3 \\
\\
& \overset{\eqref{BigGrad}}{\geq} &
\frac{\epsilon^{3/2}}{2(192)\sigma^{1/2}}
+ \frac{\sigma}{48} \|x_{t + 1} - x_{t}\|^3
- \frac{2^9 L^3}{3 \sigma^2} \|x_{t} - x_0\|^3,
\ea
\eeq
for any $0 \leq t \leq t^{\star}$.
Finally, summing up these inequalities, and using the triangle inequality, we obtain
$$
\ba{rcl}
F(x_0) - F(x_{t^{\star} + 1})
& \geq &
\frac{\epsilon^{3/2}}{2(192)\sigma^{1/2}} (t^{\star} + 1)
+ 
\frac{\sigma}{48} \sum\limits_{i = 1}^{t^{\star} + 1}
r_i^3
- \frac{2^9 L^3}{3 \sigma^2}
\sum\limits_{i = 1}^{t^{\star}} 
\Bigl(  \sum\limits_{j = 1}^i  r_j \Bigr)^3, 
\ea
$$ 
where $r_i := \|x_i - x_{i - 1}\|$.
Using Lemma B.1 from \cite{doikov2023second} and our
choice of $\sigma$ \eqref{HFSimgaH}
we conclude that
$$
\ba{rcl}
F(x_0) - F(x_{t^{\star}})
& \geq &
\frac{\epsilon^{3/2}}{2(192)\sigma^{1/2}} (t^{\star} + 1),
\ea
$$
which contradicts that inequality in Step 4 does not hold.
Hence, $t^{\star} = m$ and $\alpha = \text{\ttfamily success}$.
\end{proof}

For establishing the global convergence 
to a \textit{second-order stationary point}, we can use our procedure
with a stronger guarantee on the solution to the subrpoblem \eqref{FOMCond2}.
This is optional.
In case we use extra guarantee \eqref{FOMCond2}, 
the procedure should not be stopped in Step 3 anymore,
since then we are interested in points with \textit{both} small norm of the gradient and bounded smallest eigenvalue. 

We can justify the following analogue of Lemma~\ref{LemmaHF} when using condition \eqref{FOMCond2}:

\begin{lemma} \label{LemmaHF2}
	Consider the sequence $\{ x_t \}_{t = 1}^m$ generated by Algorithm~\ref{alg:HessianFree}
	with extra condition \eqref{FOMCond2} 
	on the inexact solution to the subproblem and without stop\footnote{Thus, $\alpha$ can be either {\ttfamily success} or {\ttfamily halt} in this case.} in Step 3.
	Then, under the conditions of Lemma~\ref{LemmaHF}, we have either
	\beq \label{HF2SmallGrad}
	\ba{rcl}
	\min\limits_{ 1 \leq t \leq m }
	\biggl[ \, \Delta_t \; \Def \;  \max\Bigl\{  
	\| \nabla f(x_t) + \psi'(x_t) \|, \;
	\frac{1}{\sigma} \bigl(  \frac{2}{3} \bigr)^{\frac{10}{3}}
	\bigl[ \xi(x_t) \big]^2
	\Bigr\} \, \biggr] & \leq & \epsilon,
	\ea
	\eeq
	or
	\beq \label{FOFuncProgress2}
	\ba{rcl}
	F(x) - F(\hat{x}) & \geq & \frac{\epsilon^{3/2}}{2 \cdot (192) \sigma^{1/2}} m.
	\ea
	\eeq
\end{lemma}
\begin{proof}
	Suppose that \eqref{HF2SmallGrad} does not hold, hence
	\beq \label{HF2BigGrad}
	\ba{rcl}
	\Delta_t & \geq & 
	\epsilon, \qquad 1 \leq t \leq m.
	\ea
	\eeq
	In view of extra inexact condition \eqref{FOMCond2},
	from Theorem~\ref{ThStep} with $\delta_g := 0$ and $z := x = x_0$ we obtain
	the following guarantee for one step:
	$$
	\ba{rcl}
	F(x_t) - F(x_{t + 1})
	& \overset{\eqref{InexactProgressSOThStep}}{\geq} &
	\frac{\sigma}{48}\|x_{t + 1} - x_t\|^3
	+ \frac{1}{192 \sigma^{1/2}} \Delta_{t + 1}^{3/2}
	- \frac{2^9}{3 \sigma^2} \delta_B^3 - \frac{2^9 L^3}{3\sigma^2}\|x_t - x_0\|^3 \\
	\\
	& \overset{\eqref{DeltaBBound}, \eqref{HF2BigGrad}}{\geq} &
	\frac{\epsilon^{3/2}}{2(192)\sigma^{1/2}}
	+ \frac{\sigma}{48} \|x_{t + 1} - x_{t}\|^3
	- \frac{2^9 L^3}{3 \sigma^2} \|x_{t} - x_0\|^3.
	\ea
	$$
	It remains to sum up these inequalities for all $0 \leq t \leq m - 1$ 
	and apply the same reasoning as in Lemma~\ref{LemmaHF}
	to get \eqref{FOFuncProgress2}.
\end{proof}

We are ready to present our whole algorithm, 
which is first-order implementation of CNM.
It uses procedure {\ttfamily CubicSteps} as the basic subroutine.

\begin{algorithm}[h!]
	\caption{\textbf{First-Order CNM}} \label{alg:FirstOrderCNM}
	\SetKwInOut{Input}{input}\SetKwInOut{Output}{output}
	\BlankLine
	\textbf{Step 0.} Given $x_0 \in Q$, $\tau_0 > 0$, $\epsilon > 0$, 
	$m \in \mathbb{N} \setminus \{ 0 \}$, set $k : = 0$.
	\BlankLine
	{\bf Step 1.}  Set $\ell := 0$.
	\BlankLine
	{\bf Step 1.1.} Using
	\beq \label{AlgFOSigm}
	\ba{rcl}
	\sigma_{k, \ell} & = & 
	2^4 \bigl( \frac{2}{3} \bigr)^{1/3} (2^{\ell} \tau_k) m
	\ea
	\eeq
	and
	\beq \label{AlgFOH}
	\ba{rcl}
	h_{k, \ell} & = & \Bigl[ 
	\frac{3 \sigma_{k, \ell}^{3/2} \epsilon^{3/2}}{2^7(192)n^{3/2} (2^{\ell} \tau_k)^3} 
	\Bigr]^{1/3}
	\ea
	\eeq
	compute $B_{k, \ell} = \frac{1}{2}\bigl(A_{k, \ell} + A_{k, \ell}^{\top} \bigr)$ with
	\beq \label{AlgFOA}
	\ba{rcl}
	A_{k, \ell} & = &
	\Bigl[  \frac{\nabla f(x_k + h_{k, \ell} e_1) - \nabla f(x_k)}{h_{k, \ell}}
	\; , \; \ldots \; , \;
	\frac{\nabla f(x_k + h_{k, \ell} e_n) - \nabla f(x_k)}{h_{k, \ell}}  \Bigr].
	\ea
	\eeq
	\BlankLine
	{\bf Step 1.2.} Perform $m$ inexact Cubic steps using the same Hessian approximation:
	$$
	\ba{rcl}
	(\hat{x}_{k, \ell}, \alpha_{k, \ell})  & := & 
	\text{\ttfamily CubicSteps}(x_k, \, B_{k, \ell}, \, \sigma_{k, \ell}, \, m, \, \epsilon).
	\ea
	$$
	\BlankLine
	{\bf Step 2.} If $\alpha_{k, \ell} = \text{\ttfamily halt}$, then 
	set $\ell := \ell + 1$ and go to Step 1.1.
	\BlankLine
	{\bf Step 3.} Set $x_{k + 1} = \hat{x}_{k, \ell}$.
	\BlankLine
	{\bf Step 4.}
	If $\alpha_{k, \ell} = \text{\ttfamily success}$,
	then $\tau_{k + 1} = \max\{ \tau_0,  \, 2^{\ell_k - 1} \tau_k\}$, $k := k + 1$, and go to Step 1. 
	Stop otherwise.
	\BlankLine
\end{algorithm}

Due to Lemmas~\ref{LemmaHF} and \ref{LemmaHF2}, this algorithm is well-defined
and its inner loop of the adaptive search (Steps 1-2) always quits 
with a sufficiently big finite value of $\ell$ and the method continues to Step~3.
In the following lemmas, we show how to bound the maximal value for the regularization parameter
and the total number of inner loop steps in our algorithm.

\begin{lemma} \label{LemmaHFTau}
	Suppose that A1 holds and let
	$\{ \tau_k \}_{k \geq 0}$ be generated by
	Algorithm~\ref{alg:FirstOrderCNM}. Then
	\beq \label{HFTauBound}
	\ba{rcl}
	\tau_k & \leq & \max\bigl\{  \tau_0, L  \bigr\},
	\qquad \forall k \geq 0.
	\ea
	\eeq
\end{lemma}
\proof
Clearly, \eqref{HFTauBound} is true for $k = 0$. 
Suppose that it is also true for some $k \geq 0$.
If $\ell_k = 0$, then it follows
from the definition of $\tau_{k + 1}$
and from the induction assumption that
$$
\ba{rcl}
\tau_{k + 1} & = & \frac{1}{2} \tau_k
\;\; < \;\; \tau_k
\;\; \leq \;\; \max\bigl\{  \tau_0, L  \bigr\},
\ea
$$
and so \eqref{HFTauBound} is true for $k + 1$.
Now, suppose that $\ell_k > 0$. In this case,
we also must have
$$
\ba{rcl}
\tau_{k + 1} & \leq & \max\bigl\{  \tau_0, L  \bigr\},
\ea
$$
since otherwise we would have
$$
\ba{rcl}
2^{\ell_k - 1} \tau_k & > & L
\ea
$$
and by \eqref{AlgFOSigm}, \eqref{AlgFOH}, \eqref{AlgFOA}
and Lemma~\ref{LemmaHF},
the inner procedure {\ttfamily CubicSteps} would return 
$\alpha_{k, \ell} \in \{ \text{\ttfamily success}, \text{ \ttfamily solution}\}$
for some $\ell \leq \ell_k - 1$,
contradicting the definition of $\ell_k$.
Thus, \eqref{HFTauBound} is also true for $k + 1$ in this case.
\qed

\begin{lemma} \label{LemmaHFNumber}
Suppose that A1 holds and let
$\text{FO}_{T}$
be the total number of function and
gradient evaluations of $f(\cdot)$ performed by Algorithm~\ref{alg:FirstOrderCNM}
during the first $T$ iteration.
Then
\beq \label{HFNumberBound}
\ba{rcl}
\text{FO}_{T}
& \leq & 
\bigl[
5 + 2(n + m)
\bigr] \cdot T
\; + \; \bigl[ 2 + n + m\bigr] \cdot
\log_2 \frac{\max\{ \tau_0, L \}}{\tau_0}.
\ea
\eeq
\end{lemma}
\proof
The total number of function and gradient evaluations
performed at the $k$the iteration of Algorithm~\ref{alg:FirstOrderCNM}
is bounded from above by
$$
\ba{rcl}
1 + \bigl[  (n + 1) + (m + 1) \bigr] \cdot (\ell_k + 1).
\ea
$$
Since $\tau_{k + 1} = 2^{\ell_k - 1} \tau_k$, we have
$$
\ba{rcl}
\ell_k - 1 & = & \log_2 \tau_{k + 1} - \log_2 \tau_k,
\ea
$$
and so
$$
\ba{rcl}
1 + \bigl[  (n + 1) + (m + 1) \bigr] \cdot (\ell_k + 1)
& = &
1 + \bigl[ 2 + n + m \bigr] \cdot (2 + \log_2 \tau_{k + 1} - \log_2 \tau_k).
\ea
$$
Thus,
$$
\ba{rcl}
\text{FO}_{T}
& \leq & 
\sum\limits_{k = 0}^{T - 1}
1 + \bigl[ 2 + n + m \bigr] \cdot (2 + \log_2 \tau_{k + 1} - \log_2 \tau_k) \\
\\
& = &
T + \bigl[ 2 + n + m\bigr] \cdot 2T
+ \bigl[ 2 + n + m\bigr] \cdot \log_2 \frac{\tau_{T}}{\tau_0} \\
\\
& \leq & 
\bigl[5 + 2(n + m) \bigr] \cdot T
+ \bigl[  2 + n + m \bigr] \cdot
\log_2 \frac{\max\{\tau_0, L  \}}{\tau_0},
\ea
$$
where the last inequality follows from Lemma~\ref{LemmaHFTau}.
\qed

We are ready to establish the global complexity bound
for our Hessian-free CNM.

\begin{theorem} \label{TheoremHF}
	Suppose that A1 holds and let $\{ x_k \}_{k \geq 1}$ be generated by
	Algorithm~\ref{alg:FirstOrderCNM}.
	Let $T(\epsilon) \leq +\infty$
	be the first iteration index such that $\| \nabla f(x_{T(\epsilon)})
	+ \psi'( x_{T(\epsilon)} ) \| \leq \epsilon$,
	for a certain $\psi'( x_{T(\epsilon)} ) \in \partial \psi( x_{T(\epsilon)} )$.
	We have
	\beq \label{HFTBound}
	\ba{rcl}
	T(\epsilon)
	& \leq & 
	\frac{(384) 2^{5/2} (\frac{2}{3})^{1/6} \max\{ \tau_0, L \}^{1/2}  (F(x_0) - F^{\star})  }{
		\sqrt{m} }
	\cdot \epsilon^{-3/2} 
	\ea
	\eeq
	and, consequently, the total number of the function and gradient evaluations is bounded as
	\beq \label{HFTotBound}
	\ba{rcl}
	\text{FO}_{T(\epsilon)}
	& \leq & 
	\frac{ [5 + 2(n + m)] }{\sqrt{m}}
	(384) 2^{5/2} \bigl(  \frac{2}{3} \bigr)^{1/6}
	\max\{ \tau_0, L \}^{1/2} ( F(x_0) - F^{\star} ) \cdot \epsilon^{-3/2} \\
	\\
	& & \qquad + \;
	[2 + n + m] \log_2 \frac{\max\{ \tau_0, L \}}{\tau_0}.
	\ea
	\eeq
\end{theorem}
\proof
By the definition of $T(\epsilon)$, we have
$$
\ba{rcl}
\| \nabla f(x_k) + \psi'(x_k) \| & \geq & \epsilon, \quad
\text{for} \quad k = 0, \ldots, T(\epsilon) - 1,
\quad
\text{and} \quad \forall \psi'(x_k) \in \partial \psi(x_k).
\ea
$$
Consequently, by Lemma~\ref{LemmaHF} we have
\beq \label{HFFuncGlob}
\ba{rcl}
F(x_k) - F(x_{k + 1}) & \geq &
\frac{\epsilon^{3/2}}{(384) \sigma_{k, \ell_k}^{1/2}}
\quad \text{for} \quad
k = 0, \ldots , T(\epsilon) - 1.
\ea
\eeq
Moreover, by Lemma~\ref{LemmaHFTau} we also have
\beq \label{HFSigmaBound}
\ba{rcl}
\sigma_{k, \ell_k} & = & 
2^4 \bigl( \frac{2}{3} \bigr)^{1/3}
m (2^{\ell_k} \tau_k)
\;\; = \;\;
2^4 \bigl(  \frac{2}{3} \bigr)^{1/3}
m (2 \tau_{k + 1})
\;\; \leq \;\;
2^5 \bigl( \frac{2}{3} \bigr)^{1/3}
m \cdot \max\{ \tau_0, L \}.
\ea
\eeq
Combining \eqref{HFFuncGlob} and \eqref{HFSigmaBound},
it follows that
$$
\ba{rcl}
F(x_k) - F(x_{k + 1})
& \geq & 
\frac{\epsilon^{3/2} \sqrt{m}}{(384) 2^{5/2} (\frac{2}{3})^{1/6} \max\{ \tau_0, L \}^{1/2}},
\quad \text{for} \quad
k = 0, \ldots, T(\epsilon) - 1.
\ea
$$
Summing up these inequalities and using
the lower bound $F^{\star}$  on $F(\cdot)$, we get
$$
\ba{rcl}
F(x_0) - F^{\star} & \geq & 
F(x_0) - F(x_{T(\epsilon)}) \\
\\
& = & 
\sum\limits_{k = 0}^{T(\epsilon) - 1}
F(x_k) - F(x_{k + 1}) \\
\\
& \geq & 
\frac{\epsilon^{3/2} \sqrt{m} }{ (384) 2^{5/2} (\frac{2}{3})^{1/6} \max\{ \tau_0, L \}^{1/2}}
T(\epsilon)
\ea
$$
which inplies \eqref{HFTBound}.
Finally, combining \eqref{HFTBound} and Lemma~\ref{LemmaHFNumber}
we obtain \eqref{HFTotBound}.
\qed

\begin{corollary}
By taking $m := n$, it follows from Theorem~\ref{TheoremHF}
that Algorithm~\ref{alg:FirstOrderCNM} needs at most
$$
\ba{c}
\mathcal{O}\bigl(   n^{1/2} \epsilon^{-3/2} + n \bigr)
\ea
$$
total function and gradient evaluations of $f(\cdot)$ to generate $x_k$
such that $\| \nabla f(x_k) + \psi'(x_k) \| \leq \epsilon$.
\end{corollary}

Let us establish a similar complexity result
for reaching the \textit{second-order} stationary points
by Algorithm~\ref{alg:FirstOrderCNM}, providing the guarantee
on the values of $\xi(\cdot)$  (see definition \eqref{XiDef}).

\begin{theorem} \label{TheoremHF2}
	Suppose that A1 holds. Let $x_{k, \ell}(t)$
	be the $t$-th iterate of Algorithm~\ref{alg:HessianFree}
	with extra condition \eqref{FOMCond2}
	and without stop in Step 3,
	applied at the $k$-th iteration
	of Algorithm~\ref{alg:FirstOrderCNM}.
	Let $T(\epsilon) \leq +\infty$ be the first iteration index such that
	\beq \label{HFSOGuarantee}
	\ba{rcl}
	\max\Bigl\{ \,
	\| \nabla f(x_{T( \epsilon ), \ell}(t) ) + \psi'(x_{T( \epsilon ), \ell}(t) ) \|,
	\;
	\frac{1}{2^2 3^3 \cdot m \cdot \max\{\tau_0, L \}}
	\bigl[ \xi(x_{T( \epsilon ), \ell}(t) ) \bigr]^2
	\,
	\Bigr\} & \leq & \epsilon,
	\ea
	\eeq
	for some $\ell \geq 0$ and $t \in \{0, \ldots, m\}$.
	Then, bounds \eqref{HFTBound} and \eqref{HFTotBound} hold.
\end{theorem}
\begin{proof}
	The proof is similar to those one of Theorem~\ref{TheoremHF},
	using Lemma~\ref{LemmaHF2} instead of Lemma~\ref{LemmaHF}.
\end{proof}

Therefore, we conclude that our Hessian-free scheme achieves
the second-order stationary guarantee \eqref{HFSOGuarantee},
even though the method does not need to compute directly 
\textit{any second-order information}, using solely the first-order oracle for $f(\cdot)$.

\section{Zeroth-Order CNM}
\label{SectionZO}

In this section, 
we present the \textit{zeroth-order} implementation
of the Cubic Newton Method, which uses only
the \textit{function evaluations} for $f(\cdot)$ to solve our
optimization problem~\eqref{MainProblem}.
Hence, we will use finite difference approximations
\textit{both} for the Hessian and for the gradients.

Note that approximating the Hessian matrix 
\eqref{HessZOBDef}
remains to be \textit{$n$ times more expensive} than the gradient vector
\eqref{GradZODef}. Therefore,
we keep using each approximation
$B_{k, \ell} \approx \nabla^2 f(x_k)$
for  consecutive $m \geq 1$  inexact cubic steps, 
while updating the gradient estimates each step.
In what follows, we show that the optimal schedule is $\boxed{m := n}$,
which gives the best zeroth-order oracle complexity
for our scheme.

Let us denote by $(\hat{x}, \alpha) = \texttt{ZeroOrderCubicSteps}(x, B, \tau, m, \epsilon)$
an auxiliary procedure that performs $m$ inexact Cubic
Newton steps~\eqref{CompositeSubproblem},
starting from a point $x \in Q$, using the same given matrix $B = B^{\top}$,
and estimating the new gradients with finite differences.
We use $\sigma > 0$ as a regularization parameter, 
and $\epsilon > 0$ is the target accuracy \eqref{InexactSolution}. The procedure
returns the last computed iterate $\hat{x}$ and a status variable
$$
\ba{rcl}
\alpha & \in & \{ \texttt{success}, \, \texttt{halt}  \},
\ea
$$ 
which indicates whether the progress condition was satisfied
for all steps or not.
We define this procedure formally as Algorithm~\ref{alg:ZO_CNM}.

\begin{algorithm}[h!]
	\caption{$\text{\ttfamily ZeroOrderCubicSteps}(x, B, \sigma, m, \epsilon)$} \label{alg:ZO_CNM}
	\SetKwInOut{Input}{input}\SetKwInOut{Output}{output}
	\BlankLine
	\textbf{Step 0.} Set $x_0 := x$ and $t := 0$.
	\BlankLine
	{\bf Step 1.} 	If $t = m$ then stop and \textbf{return} $(x_t, \, \text{\ttfamily success})$.
	\BlankLine
	{\bf Step 2.}  For
	\beq \label{ZO_hgdef}
	\ba{rcl}
	h_g & = & 
	\frac{1}{3^{1/3}} \Bigl[ \frac{\epsilon m }{\sigma n^{1/2}  }  \Bigr]^{1/2}
	\ea
	\eeq
	compute $g_t \in \R^n$ by
	\beq \label{ZO_gtdef}
	\ba{rcl}
	\bigl[ g_t \bigr]^{(i)}
	& = & 
	\frac{f(x_t + h_g e_i) - f(x_t - h_g e_i)}{2 h_g},
	\quad	i = 1, \ldots, n.
	\ea
	\eeq
	\BlankLine
	{\bf Step 3.}  Compute $x_{t + 1}$ as an approximate solution to the subproblem
	$$
	\ba{c}
	\min\limits_{y \in Q} \Bigl\{ \, M_{x_{t}, \sigma}(y) + \psi(y) \, \Bigr\},
	\qquad \text{where}
	\ea
	$$
	$$
	\ba{rcl}
	M_{x_{t}, \sigma}(y) 
	& \equiv &
	f(x_t) + \la g_t, y - x_t \ra
	+ \frac{1}{2}\la B(y - x_t), y - x_t \ra
	+ \frac{\sigma}{6}\|y - x_t\|^3
	\ea
	$$
	such that
	\beq \label{ZOMCond}
	\ba{rcl}
	M_{x_{t},\sigma}(x_{t + 1}) + \psi(x_{t + 1}) &\leq & F(x_{t}) \quad \text{and} \\
	\\ 
	\|\nabla M_{x_{t}, \sigma} ( x_{t + 1} )   +  \psi'(x_{t + 1}) \| & \leq &
	\frac{\sigma}{4} \|x_{t + 1} - x_t  \|^{2} \quad \text{for some} \;\; \psi'(x_{t + 1}) \in \partial \psi(x_{t + 1}),
	\ea
	\eeq
	and (\textbf{optionally}, if $\psi$ is twice differentiable) such that
	\beq \label{ZOMCon2}
	\ba{rcl}
	B + \sigma \|x_{t + 1} - x_t \| I + \nabla^2 \psi(x_{t + 1}) & \succeq & 0.
	\ea
	\eeq
	\BlankLine
	{\bf Step 4.} If $F(x_0) - F(x_{t + 1}) \geq \frac{\epsilon^{3/2}}{384 \sigma^{1/2}} (t + 1)$
	holds then set $t := t + 1$ and go to Step 1. Otherwise, stop and \textbf{return} $(x_{t + 1}, \text{\ttfamily halt})$.
	\BlankLine
\end{algorithm}

We can prove the following main result about this procedure.

\begin{lemma} \label{LemmaZOStep}
	Suppose that A1 holds. Given $x \in Q$, $\epsilon > 0$, $\sigma > 0$,
	and $m \in \mathbb{N} \setminus \{ 0 \}$,
	let $(\hat{x}, \alpha)$ be the corresponding output of Algorithm~\ref{alg:ZO_CNM}
	with
	$B = \frac{1}{2}(A + A^{\top})$, where
	\beq \label{ZOADef}
	\ba{rcl}
	A^{(i, j)}
	& = & 
	\frac{f(x + h e_i + h e_j) - f(x + h e_i) - f(x + h e_j) - f(x)}{h^2},
	\quad 
	i, j = 1, \ldots, n,
	\ea	
	\eeq
	for some $h > 0$. If
	\beq \label{SigmaHBZOCondition}
	\ba{rcl}
	\sigma & \geq & 2^4 \bigl( \frac{2}{3}  \bigr)^{1/3} mL
	\quad \text{and} \quad
	h \;\;  \leq \;\;
	\Bigl[  \frac{3^4 \sigma^{3/2} \epsilon^{3/2}}{ 2^{14} (192)  n^3 L^3  }  \Bigr]^{1/3},
	\ea
	\eeq
	then, for the iterations $\{ x_t \}_{t = 1}^{m}$ of Algorithm~\ref{alg:ZO_CNM}, we have either
\beq \label{ZOMinGrad}
\ba{rcl}
\min\limits_{t = 1, \ldots, m} \| \nabla f(x_t) + \psi'(x_t) \| & \leq & \epsilon,
\ea
\eeq
or
\beq \label{ZOFuncProgress}
\ba{rcl}
F(x) - F(\hat{x}) & \geq & 
\frac{\epsilon^{3/2}}{2 (192) \sigma^{1/2}} m.
\ea
\eeq
\end{lemma}
\proof
By \eqref{ZO_gtdef} and Lemma~\ref{LemmaGradZO} we have
\beq \label{ZO_gtbound}
\ba{rcl}
\| g_t - \nabla f(x_t) \| & \leq & \delta_g
\ea
\eeq
for 
\beq \label{DeltaG_ZO_expr}
\ba{rcl}
\delta_g & = & \frac{\sqrt{n} L}{6} h_g^2.
\ea
\eeq
In view of \eqref{ZO_hgdef} and the assumption \eqref{SigmaHBZOCondition}
it follows that
\beq \label{ZO_deltaG_bound}
\ba{rcl}
\frac{3}{\sigma^{1/2}} \cdot \delta_g^{3/2}
& \overset{\eqref{DeltaG_ZO_expr}}{=} &
\frac{3}{\sigma^{1/2}} \cdot \frac{n^{3/4} L^{3/2}}{6^{3/2}}
h_g^3
\;\; \overset{\eqref{ZO_hgdef}}{=} \;\;
\frac{\epsilon^{3/2}}{ 2^{8}  3  \sigma^{1/2}}
\cdot \frac{1}{\sigma^{3/2}}
\cdot \frac{ 2^{13/2}  m^{3/2} L^{3/2}}{3^{1/2}} \\
\\
& \overset{\eqref{SigmaHBZOCondition}}{\leq} &
\frac{\epsilon^{3/2}}{ 4(192)  \sigma^{1/2}}.
\ea
\eeq
On the other hand, by \eqref{ZOADef} and Lemma~\ref{LemmaHessZO} we have
\beq \label{ZO_Bbound}
\ba{rcl}
\| B - \nabla^2 f(x) \| & \leq & \delta_B
\ea
\eeq
for
$$
\ba{rcl}
\delta_B & = & \frac{2n L}{3} h.
\ea
$$
Then, in view of \eqref{SigmaHBZOCondition},
it follows that
\beq \label{ZO_deltaB_bound}
\ba{rcl}
\frac{2^9}{3 \sigma^2} \cdot \delta_B^3
& = & 
\frac{2^9}{3 \sigma^2} \cdot \frac{2^3 n^3 L^3}{3^3} \cdot h^3
\;\; \leq \;\;
\frac{2^9}{3 \sigma^2} \cdot \frac{2^3 n^3 L^3}{3^3}
\cdot \frac{3^4 \sigma^{3/2} \epsilon^{3/2} }{2^{14} (192) n^3 L^3} \\
\\
& = &
\frac{\epsilon^{3/2}}{ 4(192) \sigma^{1/2} }
\ea
\eeq
Combining \eqref{ZO_deltaG_bound} and \eqref{ZO_deltaB_bound},
we have
\beq \label{ZO_deltas_bound}
\ba{rcl}
\frac{2^9}{3 \sigma^2} \delta_B^3
+ \frac{3}{\sigma^{1/2}} \delta_g^{3/2}
& \leq &
\frac{\epsilon^{3/2}}{2(192) \sigma^{1/2}}. 
\ea
\eeq
Then, by \eqref{ZOMCond}, \eqref{ZO_gtbound},
\eqref{ZO_Bbound}, \eqref{ZO_deltas_bound}
and Theorem~\ref{ThStep}
with $z = x$, we obtain
\beq \label{ZO_func_prog}
\ba{rcl}
F(x_{t - 1}) - F(x_t) & \geq & 
\frac{\sigma}{48} \|x_t - x_{t - 1}\|^3
+ \frac{1}{192 \sigma^{1/2}} \| \nabla f(x_t)  + \psi'(x_t) \|^{3/2} \\
\\
& & \qquad
\; - \; \frac{1}{2(192)\sigma^{1/2}} \epsilon^{3/2}
- \frac{2^9 L^3}{3\sigma^2} \|x_{t - 1} - x_t\|^3,
\ea
\eeq
for $t = 1, \ldots, m$. Consequently, if \eqref{ZOMinGrad}
is not true, then
$$
\ba{rcl}
F(x_{t - 1}) - F(x_t)
& \geq & 
\frac{\sigma}{48} \|x_t - x_{t - 1} \|^3
+ \frac{1}{2(192) \sigma^{1/2}} \epsilon^{3/2}
- \frac{2^9 L^3}{3 \sigma^2} \|x_{t - 1} - x_0\|^3
\ea
$$
for $t = 1, \ldots, m$. Finally, 
summing up these inequalities and using Lemma~B.1 in \cite{doikov2023second}
for our choice \eqref{SigmaHBZOCondition} of $\sigma$,
we conclude that \eqref{ZOFuncProgress} is true.
\qed

Let us formulate our new optimization method for solving problem~\eqref{MainProblem},
which is the zeroth-order implementation of CNM.

\begin{algorithm}[h!]
	\caption{\textbf{Zero-Order CNM}} \label{alg:ZeroOrderCNM}
	\SetKwInOut{Input}{input}\SetKwInOut{Output}{output}
	\BlankLine
	\textbf{Step 0.} Given $x_0 \in Q$, $\tau_0 > 0$, $\epsilon > 0$, 
	$m \in \mathbb{N} \setminus \{ 0 \}$, set $k : = 0$.
	\BlankLine
	{\bf Step 1.}  Set $\ell := 0$.
	\BlankLine
	{\bf Step 1.1.} Using
	\beq \label{AlgZOSigm}
	\ba{rcl}
	\sigma_{k, \ell} & = & 
	2^4 \bigl( \frac{2}{3} \bigr)^{1/3} (2^{\ell} \tau_k) m
	\ea
	\eeq
	and
	\beq \label{AlgZOH}
	\ba{rcl}
	h_{k, \ell} & = & \Bigl[ 
	\frac{3^4 \sigma_{k, \ell}^{3/2} \epsilon^{3/2}}{
		2^{14}(192)n^{3} (2^{\ell} \tau_k)^3}  
	\Bigr]^{1/3}
	\ea
	\eeq
	compute $B_{k, \ell} = \frac{1}{2}\bigl(A_{k, \ell} + A_{k, \ell}^{\top} \bigr)$ with
	\beq \label{AlgZOA}
	\ba{rcl}
	\bigl[ A_{k, \ell} \bigr]^{(i, j)} & = &
	\frac{f(x_k + h_{k, \ell} e_i + h_{k, \ell} e_j ) 
		- f(x_k + h_{k, \ell} e_i)
		- f(x_k + h_{k, \ell} e_j) - f(x_k)	
	}{h_{k, \ell}^2}
	\ea
	\eeq
	for $i, j = 1, \ldots, n$.
	\BlankLine
	{\bf Step 1.2.} Perform $m$ inexact zeroth-order Cubic steps with the same Hessian approximation:
	$$
	\ba{rcl}
	(\hat{x}_{k, \ell}, \alpha_{k, \ell}) & = & 
	\texttt{ZeroOrderCubicSteps}(x_k, B_{k, \ell}, \sigma_{k, \ell}, m, \epsilon).
	\ea
	$$
	{\bf Step 2.} If $\alpha_{k, \ell} = \texttt{halt}$, then set $\ell := \ell + 1$ and go
	to Step 1.1.
	\BlankLine
	{\bf Step 3.}
	Set $x_{k + 1} = \hat{x}_{k, \ell_k}$, $\tau_{k + 1} = \max\{ \tau_0, \, 2^{\ell_k - 1} \tau_k\}$,
	$k := k + 1$, and go to Step 1. 
\end{algorithm}

Employing a stronger condition \eqref{ZOMCon2}
on the solution to the subproblem, we can also
justify the progress of our procedure in terms of the 
\textit{second-order stationarity measure}.

\begin{lemma} \label{LemmaZOStep2}
	Consider the sequence $\{ x_t \}_{t = 1}^m$ generated by Algorithm~\ref{alg:ZO_CNM}
	with extra condition \eqref{ZOMCon2} 
	on the inexact solution to the subproblem.
	Then, under the assumptions of Lemma~\ref{LemmaZOStep}, we have either
	\beq \label{ZOMinGrad2}
	\ba{rcl}
	\min\limits_{ 1 \leq t \leq m }
	\biggl[ \, \Delta_t \; \Def \;  \max\Bigl\{  
	\| \nabla f(x_t) + \psi'(x_t) \|, \;
	\frac{1}{\sigma} \bigl(  \frac{2}{3} \bigr)^{\frac{10}{3}}
	\bigl[ \xi(x_t) \big]^2
	\Bigr\} \, \biggr] & \leq & \epsilon,
	\ea
	\eeq
	or
	\beq \label{ZOFuncProgress2}
	\ba{rcl}
	F(x) - F(\hat{x}) & \geq & \frac{\epsilon^{3/2}}{2 \cdot (192) \sigma^{1/2}} m.
	\ea
	\eeq
\end{lemma}
\proof
The proof follows the reasoning of Lemma~\ref{LemmaZOStep}, 
using the stronger one step guarantee provided by Theorem~\ref{ThStep}.
\qed

\begin{lemma} \label{LemmaZOTau}
	Suppose that A1 holds and let $\{ \tau_k \}_{k \geq 0}$
	be generated by Algorithm~\ref{alg:ZeroOrderCNM}. Then
	\beq \label{AlgZOTauKBound}
	\ba{rcl}
	\tau_k & \leq & \max\{ \tau_0, L \}, \qquad \forall k \geq 0.
	\ea
	\eeq
\end{lemma}
\proof
It follows exactly as in the proof of Lemma~\ref{LemmaHFTau},
using Lemma~\ref{LemmaZOStep} to conclude that
$$
\ba{rcl}
\tau_{k + 1} & \leq & \max\{ \tau_0, L \}
\ea
$$
when $\ell_k > 0$.
\qed

\begin{lemma} \label{LemmaZONumber}
	Suppose that A1 holds and let $\text{ZO}_{T}$
	be the total number of function evaluations of $f(\cdot)$ performed
	by Algorithm~\ref{alg:ZeroOrderCNM} during the first $T$ iterations.
	Then,
	$$
	\ba{rcl}
	\text{ZO}_{T} & \leq & 
	\bigl[4 + 4 mn + 6n^2\bigr] \cdot T
	\; + \;
	\bigl[2 + 2mn + 3n^2\bigr] \cdot \log_2\frac{\max\{ \tau_0, L \}}{\tau_0}.
	\ea
	$$
\end{lemma}
\proof
The number of function evaluations performed by Algorithm~\ref{alg:ZeroOrderCNM}
(including those ones performed by Algorithm~\ref{alg:ZO_CNM} in Step 2)
is bounded from above by
$$
\ba{rcl}
[2 + 2 mn + 3n^2] \cdot (\ell_k + 1).
 \ea
$$
Since $\tau_{k + 1} = 2^{\ell_k - 1} \tau_k$, we have
$$
\ba{rcl}
\ell_{k + 1} & = & 2 + \log_2 \tau_{k + 1} - \log_2 \tau_k.
\ea
$$ 
Thus,
$$
\ba{rcl}
\text{ZO}_{T} & \leq & 
\sum\limits_{k = 0}^{T - 1}
[2 + 2mn + 3n^2] \cdot (2 + \log_2 \tau_{k + 1} - \log_2 \tau_k) \\
\\
& = & 
[2 + 2mn + 3n^2] \cdot ( 2T + \log_2 \tau_T - \log_2 \tau_0 ) \\
\\
& \leq & 
[2 + 2mn + 3n^2] \cdot ( 2T + \log_2 \frac{\max\{ \tau_0, L \}}{\tau_0},
\ea
$$
where the last inequality follows from Lemma~\ref{LemmaZOTau}.
\qed

We prove the following main result.

\begin{theorem} \label{TheoremZO}
	Suppose that A1.
	Let $x_{k, \ell}(t)$ be the $t$-th iterate
	of Algorithm~\ref{alg:ZO_CNM} applied at the $k$-th iteration
	of Algorithm~\ref{alg:ZeroOrderCNM} in the $\ell$-th inner loop.
	Let $T(\epsilon) \leq +\infty$ be the first iteration index such that
	$$
	\ba{rcl}
	\| \nabla f( x_{T(\epsilon), \ell}(t) )
	    + \psi'( x_{T(\epsilon), \ell}(t) ) \| & \leq & \epsilon
	\ea
	$$
	for some $\ell \geq 0$ and $t \in \{ 0, \ldots, m \}$.
	Then,
	\beq \label{ZOTBound}
	\ba{rcl}
	T(\epsilon) & \leq & 
	\frac{(384) 2^{5/2} (\frac{2}{3})^{1/6} \max\{ \tau_0, L \}^{3/2} (f(x_0) - f^{\star}) }{\sqrt{m}}
	\epsilon^{-3/2}
	\ea
	\eeq
	and, consequently, the total number of the function evaluations is bounded as
	\beq \label{ZOTBoundTotal}
	\ba{rcl}
	\text{ZO}_{T(\epsilon)}
	& \leq & 
	\mathcal{O}\Bigl(  \frac{mn + n^2}{\sqrt{m}} \max\{ \tau_0, L \}^{1/2}
	 (f(x_0) - f^{\star}) \cdot \epsilon^{-3/2} + (mn + 3n^2) 
	 \log_2 \frac{\max\{ \tau_0, L \}}{\tau_0}   \Bigr).	
	\ea
	\eeq
\end{theorem}
\proof
Similarly to the proof of Theorem~\ref{TheoremHF},
we get \eqref{ZOTBound} from Lemma~\ref{LemmaZOStep} and Lemma~\ref{LemmaZOTau}.
Then, combining \eqref{ZOTBound} with Lemma~\ref{LemmaZONumber}, we get \eqref{ZOTBoundTotal}.
\qed

\begin{corollary}
By taking $\boxed{m := n}$, it follows from Theorem~\ref{TheoremZO} that Algorithm~\ref{alg:ZeroOrderCNM}
needs at most 
$$
\ba{c}
\mathcal{O}( n^{3/2} \epsilon^{-3/2} )
\ea
$$
function evaluations of $f(\cdot)$ to find a point $\bar{x}$ such that
$\| \nabla f(\bar{x}) + \psi'(\bar{x}) \| \leq \epsilon$.
\end{corollary}

Finally, we can establish the convergence result in terms of the \textit{second-order stationary point}.
The proof is identical and it just needs to replace Lemma~\ref{LemmaZOStep} by Lemma~\ref{LemmaZOStep2}.

\begin{theorem}
Suppose that A1 holds.
Let $x_{k, \ell}(t)$ be the $t$-th iterate
of Algorithm~\ref{alg:ZO_CNM} 
with extra condition~\eqref{ZOMCon2} 
on the inexact solution to the subproblem, applied at the $k$-th iteration
of Algorithm~\ref{alg:ZeroOrderCNM} in the $\ell$-th inner loop.
Let $T(\epsilon) \leq +\infty$ be the first iteration index such that
$$
\ba{rcl}
\max\Bigl\{ \,
\| \nabla f( x_{T(\epsilon), \ell}(t) ) + \psi'(x_{T(\epsilon), \ell}(t))  \|, \;
\frac{1}{2^2 3^3 \cdot m \cdot \max\{ \tau_0, L \}}
\bigl[ \xi( x_{T(\epsilon), \ell}(t) )  \bigr]^2
\,
\Bigr\} 
& \leq & \epsilon
\ea
$$
for some $\ell \geq 0$ and $t \in \{ 0, \ldots, m \}$.
Then, bounds \eqref{ZOTBound} and \eqref{ZOTBoundTotal} hold.
\end{theorem}

\section{Local Superlinear Convergence}
\label{SectionLocal}

One of the main classical results about Newton's Method is
its \textit{local quadratic convergence}, which dates back
to the works of Fine \cite{fine1916newton}, Bennett \cite{bennett1916newton}, 
and Kantorovich \cite{kantorovich1948newton}.
It assumes that the iterates of the method are already in a neighbourhood
of a non-degenerate solution (a strict local minimum $x^{\star}$ satisfying $\nabla^2 f(x^{\star}) \succ 0$),
and it shows importantly that under this condition the method converges very fast.

Later on, a local superlinear convergence of the Newton Method that uses the same Hessian for $m \geq 1$
consecutive steps, where $m$ is a parameter, was established by Shamanskii in \cite{shamanskii1967modification},
and recently in \cite{doikov2023second}. 
The local quadratic convergence 
of the CNM with finite difference Hessian approximations was studied in \cite{grapiglia2022cubic}.

In this section, we justify local superlinear convergence for our implementations of the inexact composite CNM.
To quantify our problem class, we additionally assume the following\footnote{Note that for simplicity we assume here strong convexity for the whole feasible set $Q$, while it can be possible to restrict our analysis to a neighbourhood of a non-degenerate local minimum.}:

\textbf{A2} The Hessian of $f$ is below bounded on $Q$, for some $\mu > 0$:
\beq \label{HessBounded}
\ba{rcl}
\nabla^2 f(x) & \succeq & \mu I, \qquad \forall x \in Q.
\ea
\eeq

It is well known that bound \eqref{HessBounded} means that our composite objective $F(\cdot)$
is \textit{strongly convex} on~$Q$ with parameter $\mu > 0$.
Thus, it has unique minimizer $x^{\star} \in Q$,
and the following standard inequality holds \cite{nesterov2018lectures}:
\beq \label{SolBound}
\ba{rcl}
\| x - x^{\star} \| & \leq & \frac{1}{\mu} \| F'(x) \|, \qquad \forall x \in Q,  \; F'(x) \in \partial F(x).
\ea
\eeq

Let us study one iteration $k \geq 0$ of our first-order CNM (Algorithm~\ref{alg:FirstOrderCNM}).
First, we have the following bounds, for any $\ell \geq 0$:
\beq \label{LocalSigmaHBound}
\ba{rcl}
\sigma_{k, \ell}
& \overset{\textit{Step~1.1}}{=} &
2^4 \bigl(  \frac{2}{3} \bigr)^{1/3} (2^{\ell} \tau_k) m
\;\; \overset{ \textit{Step~4}, \, \eqref{HFTauBound} }{\leq} \;\; 
2^5 \bigl(  \frac{2}{3} \bigr)^{1/3} m\cdot \max\{ \tau_0, L \}, \\
\\
h_{k, \ell} & \overset{\textit{Step~1.1}}{=} &
\Bigl[  \frac{3 \cdot 2^6 ( \frac{2}{3} )^{1/3} m^{3/2} \epsilon^{3/2} }{
				   2^7 (192) n^{3/2} (2^\ell \tau_k)^{3/2}
				}   \Bigr]^{1/3}
\;\; = \;\;
c \cdot \sqrt{\frac{m \epsilon}{2^{\ell} n \tau_k } } \\
\\
& \leq & 
c \cdot \sqrt{\frac{m \epsilon}{n \tau_k } } 
\;\; \overset{\textit{Step~4}}{\leq} \;\;
c \cdot \sqrt{\frac{m \epsilon}{n \tau_0 } },
\ea
\eeq
where $c := \frac{1}{3^{1/6} \cdot 2^{13/6}} $
is a numerical constant.

Let us consider the following set, for a fixed $\epsilon, \kappa > 0$
and some given selection of subgradients $F'(x) \in \partial F(x)$:
\beq \label{BoundSubgr}
\ba{rcl}
\mathcal{Q}_{\epsilon, \kappa} & \Def &
\Bigl\{ \,
x \in Q \; : \; \epsilon \, \leq \, \| F'(x) \| \, \leq \, \frac{\kappa }{2} \, \Bigr\},
\ea
\eeq
where $\kappa$ is the following constant describing the \textit{region of quadratic convergence}:
\beq \label{KappaDef}
\ba{rcl}
\kappa & \Def & \mu^2  \cdot \frac{1}{2} \biggl[
3 \cdot 2^6 \bigl( \frac{2}{3} \bigr)^{1/3} m \cdot \max\{\tau_0, L  \}
+ 8L + \frac{c^2 L^2 m}{\tau_0}
\biggr]^{-1}
\;\; \sim \;\;\; \frac{\mu^2}{mL}.
\ea
\eeq
By \eqref{BoundSubgr}, we assume the desired accuracy $\epsilon$ to be 
sufficiently small:
\beq \label{LocalEpsBound}
\ba{rcl}
\epsilon & \leq & \frac{\kappa}{2} 
\;\; \leq \;\; \frac{\tau_0 \mu^2}{m L^2 c^2}.
\ea
\eeq
Then, we have
\beq \label{LocalHklBound}
\ba{rcl}
h_{k, \ell} & \overset{\eqref{LocalSigmaHBound}}{\leq} &
c \cdot \sqrt{ \frac{m \epsilon}{n \tau_0} }
\;\; \overset{\eqref{LocalEpsBound}}{\leq} \;\;
\frac{\mu}{L\sqrt{n}}.
\ea
\eeq
Therefore, due to Lemma~\ref{LemmaHessFO}, for all Hessian approximations $B_{k, \ell}$
constructed in Algorithm~\ref{alg:FirstOrderCNM}, it holds:
\beq \label{LocalBBound}
\ba{rcl}
\| B_{k, \ell} - \nabla^2 f(x_k)  \| & \overset{\eqref{HessFOBound}}{\leq} & 
\frac{\sqrt{n} L}{2} h_{k, \ell}
\;\; \overset{\eqref{LocalHklBound}}{\leq} \;\; \frac{\mu}{2}.
\ea
\eeq
Taking into account our assumption A2, we conclude that
our Hessian approximations are \textit{always positive definite}:
\beq \label{LocalBStrongConvex}
\ba{rcl}
B_{k, \ell} & \succeq & \frac{\mu}{2} I.
\ea
\eeq
In this case, we can easily bound the lenght of one inexact CNM step, as follows.

\begin{lemma} \label{LemmaStepBound}
	Let $x^+$ be an inexact minimizer of model
	\eqref{ApproxModel} satisfying the following condition:
	\beq \label{InexactThStep2}
	\ba{rcl}
	\| \nabla M_{x, \sigma}(x^+) + \psi'(x^+) \| & \leq & \frac{\sigma}{4}\|x^+ - x\|^2,
	\ea
	\eeq
	for a certain $\psi'(x^+) \in \partial \psi(x^+)$,
	where $g$ satisfy \eqref{GradHessApprox} for some $\delta_g \geq 0$,
	and $B \succeq \frac{\mu}{2}I$.  Then, we have
	\beq \label{CubicStepBound}
	\ba{rcl}
	r \;\; := \;\; \|x^+ - x \| & \leq & 
	\frac{2}{\mu}\Bigl(  \| F'(x) \| + \delta_g  \Bigr),
	\qquad \forall F'(x) \in \partial F(x).
	\ea
	\eeq
\end{lemma}
\proof
Indeed, we get that
\beq \label{StepLemma1}
\ba{rcl}
\frac{\sigma r^3}{4} & \overset{\eqref{InexactThStep2}}{\geq} &
\la \nabla M_{x, \sigma}(x^+) + \psi'(x^+), x^+ - x \ra \\
\\
& = &
\la g + B(x^+ - x) + \frac{\sigma}{2}r(x^+ - x) + \psi'(x^+), x^+ - x \ra \\
\\
& \geq & \la g + \psi'(x^+), x^+ - x \ra
+ \frac{\mu r^2}{2} + \frac{\sigma r^3}{2}.
\ea
\eeq
Hence, rearranging the terms and using convexity of $\psi$, we obtain,
for any $\psi'(x) \in \partial \psi(x)$,
$$
\ba{rcl}
\frac{\mu r^2}{2} & \overset{\eqref{StepLemma1}}{\leq} & 
\la g + \psi'(x^+), x - x^+ \ra
- \frac{\sigma r^3}{4}
\;\; \leq \;\;
\la g + \psi'(x), x - x^+ \ra - \frac{\sigma r^3}{4} \\
\\
& \leq & r \Bigl( \| \nabla f(x) + \psi'(x) \| + \delta_g \Bigr)
\;\; = \;\;
r \Bigl( \| F'(x) \| + \delta_g \Bigr),
\ea
$$
which is \eqref{CubicStepBound}.
\qed

Now, let us look at the local progress given by one inexact CNM step $x \mapsto x^+$,
with anchor point $z := x_k$.
Assuming that $x \in \mathcal{Q}_{\epsilon, \kappa}$
and under assumptions of Lemma~\ref{LemmaNewGrad}
with $\theta = \frac{\sigma_{k, \ell} }{4}$, $\delta_g = 0$, and 
$\delta_B  \overset{\eqref{HessFOBound}}{=} 
\frac{\sqrt{n} L}{2} h_{k, \ell}$, 
we get
\beq \label{LocalOneStepProgress}
\ba{rcl}
\| F'(x^+) \| & \overset{\eqref{NewGradBound} }{\leq} & 
\Bigl( \frac{3}{4} \sigma_{k, \ell} + \frac{L}{2}  \Bigr)
r^2
+ \Bigl(  \frac{\sqrt{n} L}{2} h_{k, \ell}  + L\|x -  x_k \|\Bigr)
r \\
\\
& \overset{\eqref{LocalSigmaHBound}}{\leq} &  
\Bigl( \frac{3}{4} \sigma_{k, \ell} + \frac{L}{2}  \Bigr)
r^2 
+ \Bigl( \frac{cL}{2}\sqrt{\frac{m}{\tau_0}} \cdot \sqrt{\epsilon}
+ L\|x - x_k \| \Bigr) r \\
\\
& \leq & 
\Bigl( \frac{3}{4} \sigma_{k, \ell} + \frac{L}{2} + \frac{c^2 L^2 m}{8 \tau_0}  \Bigr)
r^2 
+ L r \|x - x^{\star} \|
+ L r \|x_k - x^{\star} \| + \frac{\epsilon}{2}
\\
\\
& \overset{\eqref{CubicStepBound}, \eqref{SolBound}}{\leq} &
\frac{1}{\mu^2} \Bigl( 3 \sigma_{k, \ell} + 4L + \frac{c^2 L^2 m}{2\tau_0}  \Bigr) \| F'(x) \|^2
+ \frac{2L}{\mu^2} \| F'(x) \| \cdot \| F'(x_k) \| + \frac{\epsilon}{2}.
\ea
\eeq
We see that the first term in the right hand side of \eqref{LocalOneStepProgress}  is responsible for the local
quadratic convergence in terms of the (sub)gradient norm, as in the classical Newton's Method,
and the last two terms appear due to the inexactness of our Hessian approximations.
It remains to combine all our observations together.

\begin{theorem}
	Suppose that A1 and A2 hold. Let $x_0 \in \mathcal{Q}_{\epsilon, \kappa}$,
	given by \eqref{BoundSubgr} and $\kappa$ given by \eqref{KappaDef}. Let $\{ x_k \}_{k \geq 1}$
	be generated by Algorithm~\ref{alg:FirstOrderCNM} and denote by $T(\epsilon) \leq +\infty$
	be the first iteration index such that $\| \nabla f( x_{T(\epsilon)} ) + \psi'(x_{T(\epsilon)}) \| \leq \epsilon$,
	for a certain $\psi'(x_{T(\epsilon)}) \in \partial \psi(x_{T(\epsilon)})$.
	We have
	\beq \label{LocalTBound}
	\ba{rcl}
	T(\epsilon) & \leq & \frac{1}{\log_2(1 + m)} \log_2 \log_2 \frac{\kappa}{\epsilon} + 1.
	\ea
	\eeq
	\proof
	By the definition of $T(\epsilon)$, we have
	$$
	\ba{rcl}
	\| F'(x_k) \| \;\; \equiv \;\;
	\| \nabla f(x_k) + \psi'(x_k) \| & \geq & \epsilon, \quad \text{for} \quad
	k = 0, \ldots, T(\epsilon) - 1,
	\ea
	$$
	and for all iterations generated by Algorithm~\ref{alg:HessianFree}
	launched from Algorithm~\ref{alg:FirstOrderCNM}.
	We prove by induction that
	\beq \label{LocalGkBound}
	\ba{rcl}
	\frac{1}{\kappa} \| F'(x_k) \| & \leq & \bigl( \frac{1}{2}  \bigr)^{(1 + m)^k + 1},
	\qquad k = 0, \ldots, T(\epsilon) - 1,
	\ea
	\eeq
	which immediately leads to the desired bound.
	
	For $k = 0$, inequality \eqref{LocalGkBound} holds due to our assumption: 
	$x_0 \in \mathcal{Q}_{\epsilon, \kappa}$, and this is the base of our induction.
	Assume that it holds for some $k \geq 0$, and consider one iteration of Algorithm~\ref{alg:FirstOrderCNM}.
	In \textit{Step 1.2} it runs \texttt{CubicSteps} (Algorithm~\ref{alg:HessianFree})
	and will do the adaptive search until gets status $\alpha_{k, \ell} = \texttt{success}$
	($\alpha_{k, \ell} = \texttt{solution}$ is impossible by our assumption).
	
	Hence, $x_{k + 1}$ will be computed as $m$ inexact Cubic steps performed from 
	the point $x_k$. Denoting these steps by $x_k^{0} \mapsto x_k^{1} \mapsto \ldots \mapsto x_k^{m}$
	($x_k^{0} \equiv x_k$ and $x_k^{m} \equiv x_{k + 1}$), we conclude that, for each $0 \leq t \leq m - 1$:
	\beq \label{LocalProgress2}
	\ba{rcl}
	\| F'(x_k^{t + 1}) \| & \overset{\eqref{LocalOneStepProgress}, \eqref{LocalSigmaHBound}}{\leq} & 
	\frac{1}{2\kappa} \Bigl( \| F'(x_k^t) \|^2 + g_k \| F'(x_k^t) \| \Bigr) + \frac{\epsilon}{2} \\
	\\
	& \leq & \frac{1}{2\kappa} \Bigl( \| F'(x_k^t) \|^2 + \| F'(x_k^0) \|  \| F'(x_k^t) \| \Bigr) + \frac{1}{2} \| F'(x_k^{t + 1}) \|.
	\ea
	\eeq
	Now, assuming that 
	\beq \label{Induct2}
	\ba{rcl}
	\frac{1}{\kappa} \| F'(x_k^t) \| & \leq & \bigl( \frac{1}{2}  \bigr)^{(1 + t)(1 + m)^k + 1}
	\ea
	\eeq
	(which holds for $t = 0$ by \eqref{LocalGkBound}),
	we have
	$$
	\ba{cl}
	& \frac{1}{\kappa} \| F'(x_k^{t + 1}) \| \;\; \overset{\eqref{LocalProgress2}}{\leq} \;\;
	\frac{1}{\kappa} \Bigl(  \| F'(x_k^t) \| + \|F'(x_k^0)  \| \Bigr) \cdot \frac{1}{\kappa} \| F'(x_k^t) \| \\
	\\
	& \overset{\eqref{Induct2}}{\leq} \;\;
	\Bigl( 
	\bigl( \frac{1}{2} \bigr)^{ (1 + t)(1 + m)^k + 1} + \bigl( \frac{1}{2} \bigr)^{ (1 + m)^k + 1} 
	\Bigr) \cdot \bigl(\frac{1}{2} \bigr)^{ (1 + t)(1 + m)^k + 1} 
	\;\; \leq \;\;
	\bigl(\frac{1}{2} \bigr)^{ (1 + t + 1)(1 + m)^k + 1}.
	\ea
	$$
	Thus, \eqref{Induct2} holds for all $0 \leq t \leq m$, and for $t = m$ it gives \eqref{LocalGkBound} 
	for the next iterate.
	\qed
\end{theorem}

Finally, let us discuss the local superlinear convergence 
for our derivative-free CNM (Algorithm~\ref{alg:ZeroOrderCNM}),
while the analysis remains similar to the Hessian-free version.
For the derivative-free method, we have, for a fixed iteration $k \geq 0$
and for any $\ell \geq 0$:
\beq \label{ZOHBounds}
\ba{rcl}
\sigma_{k, \ell} & \overset{\textit{Step 1.1}}{=} & 2^4 \bigl( \frac{2}{3} \bigr)^{1/3} (2^\ell \tau_k) m
\;\; \overset{\textit{Step 3}, \;\eqref{AlgZOTauKBound}}{\leq} \;\;
2^5 \bigl(\frac{2}{3} \bigr)^{1/3} m \cdot \max\{ \tau_0, L \}, \\
\\
h_{k, \ell} 
& \overset{\textit{Step 1.1}}{=} &
\Bigl[  \bigl(\frac{2}{3}\bigr)^{1/2} 
\frac{3^3 \epsilon^{3/2} m^{3/2} }{2^{14} n^3 (2^{\ell} \tau_k)^{3/2}} \Bigr]^{1/3}
\;\; \overset{\textit{Step 3}}{\leq} \;\;  \frac{c_B \epsilon^{1/2} m^{1/2}}{n \tau_0^{1/2}},
\quad \text{where} \quad c_B \; := \; \frac{3}{2^{14/3}} \cdot \bigl( \frac{2}{3} \bigr)^{1/6},
\ea
\eeq
and in each call of Algorithm~\ref{alg:ZO_CNM}, the gradient finite difference parameter is
\beq \label{ZOHGBound}
\ba{rcl}
h_g & \overset{\textit{Step 2}}{=} &  
\frac{1}{3^{1/3}} \Bigl[  \frac{\epsilon m}{\sigma_{k,\ell} n^{1/2}} \Bigr]^{1/2}
\;\; \leq \;\; c_g \cdot \frac{\epsilon^{1/2}}{n^{1/4} \tau_0^{1/2}},
\quad \text{where} \quad c_g \; := \; \frac{1}{3^{1/3} 2^2 (2/3)^{1/6}}.
\ea
\eeq
Therefore, due to Lemma~\ref{LemmaGradZO} and \ref{LemmaHessZO},
all our gradient and Hessian approximations used in Algorithms~\ref{alg:ZO_CNM} and~\ref{alg:ZeroOrderCNM} 
satisfy the following guarantees:
\beq \label{GradHessGuarantees}
\ba{rcl}
\| g_t - \nabla f(x_t)\| & \overset{\eqref{GradZOBound}}{\leq} &
\frac{\sqrt{n} L}{6} h_g^2
\;\; \overset{\eqref{ZOHGBound}}{\leq} \;\;
\frac{c_g^2}{6} \cdot \frac{ \epsilon L}{\tau_0},
\\
\\
\| B_{k, \ell} - \nabla^2 f(x_k)  \| & \overset{\eqref{HessZOBBound}}{\leq} & 
\frac{2nL}{3} h_{k,\ell}
\;\; \overset{\eqref{ZOHBounds}}{\leq} \;\;
\frac{2 c_B}{3} \cdot \frac{\epsilon^{1/2} m^{1/2} L}{\tau_0^{1/2}}.
\ea
\eeq
In particular, assuming that $\epsilon$ is sufficiently small \eqref{LocalEpsBound},
we ensure $\| B_{k, \ell} - \nabla^2 f(x_k) \| \leq \frac{\mu}{2}$, and hence
our Hessian approximations are positive definite: $B_{k, \ell} \succeq \frac{\mu}{2} I$.

Let us assume that the initial regularization parameter is sufficiently big:
\beq \label{Tau0Big}
\ba{rcl}
\tau_0 & \geq & \frac{2 c_g^2 L}{3}.
\ea
\eeq
Using Lemma~\ref{LemmaStepBound}, we can bound one (zeroth-order) inexact CNM step $x \mapsto x^+$
for a point $x \in Q_{\epsilon, \kappa}$, as follows:
\beq \label{ZORBound}
\ba{rcl}
r & := & \|x^+ - x\| 
\;\; \overset{\eqref{CubicStepBound}, \eqref{GradHessGuarantees}}{\leq} \;\;
\frac{2}{\mu} \Bigl(  \| F'(x)\| + \frac{c_g^2}{6} \cdot \frac{\epsilon L}{\tau_0}  \Bigr) \\
\\
& \overset{\eqref{Tau0Big}}{\leq} &
\frac{2}{\mu} \Bigl(  \| F'(x)\|  + \frac{\epsilon}{4} \Bigr)
\;\; \overset{\eqref{BoundSubgr}}{\leq} \;\;
\frac{5}{2\mu} \|F'(x) \|.
\ea
\eeq
It remains to apply Lemma~\ref{LemmaNewGrad}
with $\theta = \frac{\sigma_{k,\ell}}{4}$, $\delta_g =\frac{c_g^2 \epsilon L}{6\tau_0}  \overset{\eqref{Tau0Big}}{\leq} \frac{\epsilon}{4}$,
$\delta_B = \frac{2C_B \epsilon^{1/2}m^{1/2}L}{3\tau_0^{1/2}}$
and anchor point $z := x_k$. We obtain
\beq \label{ZOLocalGrad}
\ba{rcl}
\| F'(x^+) \| & \overset{\eqref{NewGradBound}}{\leq} &
\Bigl( \frac{3}{4} \sigma_{k, \ell} + \frac{L}{2}  \Bigr) r^2
+ \Bigl(  \frac{2c_B}{3} \cdot \frac{\epsilon^{1/2} m^{1/2} L}{\tau_0^{1/2}} r + L\|x - x_k\|   \Bigr)r
+ \frac{\epsilon}{4} \\
\\
& \leq & 
\Bigl(  \frac{3}{4} \sigma_{k, \ell} + \frac{L}{2} + \frac{4c_B^2 mL^2}{9\tau_0}   \Bigr) r^2
+ Lr\|x - x^{\star} \| + Lr\|x_k - x^{\star}\| + \frac{\epsilon}{2} \\
\\
& \overset{\eqref{ZORBound}, \eqref{SolBound}}{\leq} &
\frac{1}{\mu^2}
\Bigl( 
\frac{75}{4} \sigma_{k, \ell} + \frac{45L}{8}
+ \frac{4 c_B^2 mL^2}{9\tau_0}
\Bigr) \| F'(x) \|^2
+ \frac{5L}{2\mu^2} \|F'(x) \| \cdot \| F'(x_k) \| + \frac{\epsilon}{2}.
\ea
\eeq
We see that this inequality has the same structure
as \eqref{LocalOneStepProgress} established for the Hessian-free CNM.
Applying bound \eqref{ZOHBounds}, it is easy to verify that we can use the same local region,
given by \eqref{BoundSubgr}, \eqref{KappaDef}.
Therefore, repeating the previous reasoning, we prove the following local superlinear convergence.

\begin{theorem}
	Suppose that A1 and A2 hold. Let $x_0 \in \mathcal{Q}_{\epsilon, \kappa}$,
	given by \eqref{BoundSubgr} and $\kappa$ given by \eqref{KappaDef}. 
	Let initial regularization parameter $\tau_0$ be sufficiently big~\eqref{Tau0Big}.
	Let $x_{k, \ell}(t)$ be the $t$-th iterate of Algorithm~\ref{alg:ZO_CNM}
	applied at the $k$-th iteration of Algorithm~\ref{alg:ZeroOrderCNM} 
	in the $\ell$-th inner loop. Let $T(\epsilon) \leq +\infty$ be the first
	iteration index such that
	$\| \nabla f( x_{T(\epsilon), \ell} ) + \psi'(x_{T(\epsilon), \ell}) \| \leq \epsilon$,
	for some $\ell \geq 0$ and  $t \in \{0, \ldots\, m\}$. Then,
	\beq \label{LocalTBoundZO}
	\ba{rcl}
	T(\epsilon) & \leq & \frac{1}{\log_2(1 + m)} \log_2 \log_2 \frac{\kappa}{\epsilon} + 1.
	\ea
	\eeq
\end{theorem}

\section{Illustrative Numerical Experiments}
\label{SectionExperiments}

We performed preliminary numerical experiments with Matlab implementations of the proposed methods applied to the set of 35 problems from the Mor\'e-Garbow-Hillstrom collection \cite{more1981testing}\footnote{For each problem, $n$ was chosen as in \cite{birgin2020use}, resulting in a set of problems with dimensions ranging from $2$ to $40$.} For both algorithms, we considered $\tau_{0}=1$ and $\epsilon=10^{-4}$, allowing a maximum of $3,000$ calls of the oracle. Moreover, each cubic subproblem was approximately solved by a BFGS method with Armijo line search (using the origin as initial point).

Figure~\ref{fig:01} presents the performance profiles \cite{dolan2002benchmarking}\footnote{The performance profiles were generated using the code \textbf{perf.m} freely available in the website \url{https://www.mcs.anl.gov/~more/cops/}.} for Algorithm~\ref{alg:FirstOrderCNM}, comparing the variants with $m=1$, $m=n$ and $m=2n$ in terms of the number of calls of the oracle required to find the first $\epsilon$-approximate stationary point. For each value $x$ in x-axis, we show in y-axis percentage of the problems for which the corresponding code performs with a factor $2^x$ of the best performance among all the methods. 
In accordance with our theory, $m=n$ resulted in the best performance, with the corresponding code requiring less calls of the oracle in $48.6\%$ of the problems.
\begin{figure}[h!]
	\centering
	\includegraphics[scale=0.43]{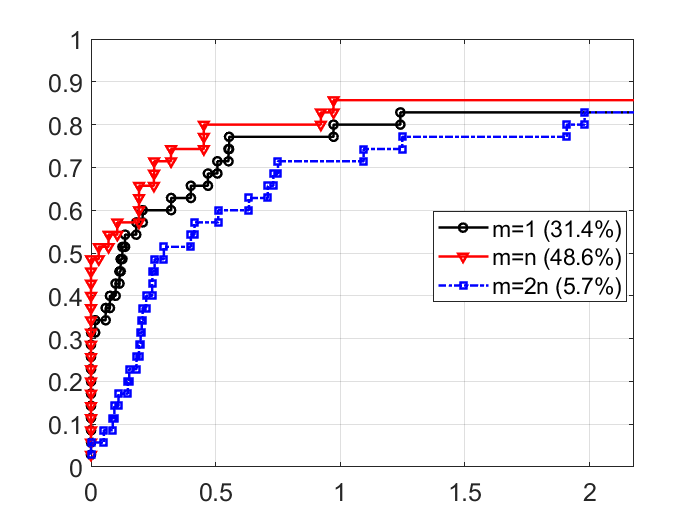}
	\caption{Performance profiles in $\log_{2}$ scale for Algorithm~\ref{alg:FirstOrderCNM}. For each choice of $m$, the caption indicates the percentage of problems in which the corresponding code was the best in terms of number of calls of the oracle.}
	\label{fig:01}
\end{figure}

We performed similar experiments with Algorithm~\ref{alg:ZeroOrderCNM}, comparing the choices $m=1$, $m=n$ and $m=2n$ in terms of the number of function evaluations required to find $\bar{x}$ such that
\begin{equation*}
\ba{rcl}
f(\bar{x})-f_{best} & \leq & \epsilon\left(f(x_{0})-f_{best}\right).
\ea
\end{equation*}
For each problem, $f_{best}$ is the smallest value of the objective function obtained 
by applying the three variants of Algorithm~\ref{alg:ZeroOrderCNM} with a budget of $3,000$ function evaluations. Figure~\ref{fig:02} presents the corresponding performance profiles. Again, the variant with $m=n$ outperformed the others, requiring less function evaluations in $60.0\%$ of the problems.
\begin{figure}[h!]
	\centering
	\includegraphics[scale=0.43]{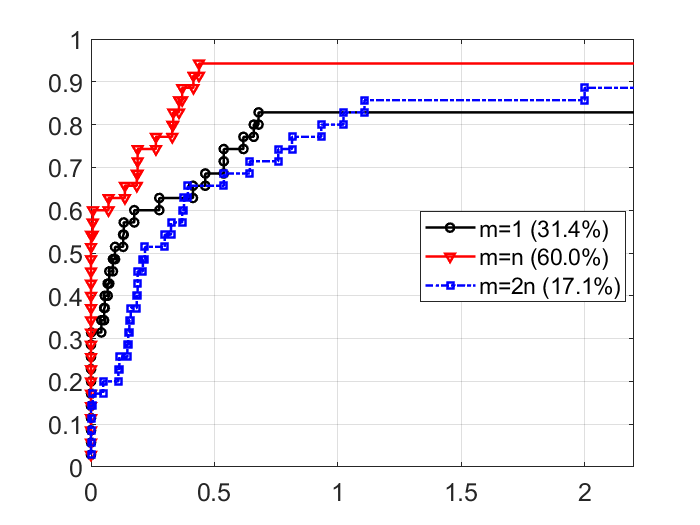}
	\caption{Performance profiles in $\log_{2}$ scale for Algorithm~\ref{alg:ZeroOrderCNM}.}
	\label{fig:02}
\end{figure}

\section{Discussion}
\label{SectionDiscussion}

In this paper, we have developed new first-order
and zeroth-order 
implementations of the Cubically regularized Newton method,
that need, respectively, at most $\mathcal{O}( n^{1/2} \epsilon^{-3/2} )$
and $\mathcal{O}( n^{3/2} \epsilon^{-3/2} )$
calls of the oracle to find an $\epsilon$-approximate second-order stationary point.
Along with improved complexity guarantees, one of the main advantages of our schemes
is the adaptive search, which makes the algorithms free from the need to fix the actual Lipschitz constant
and the finite-difference approximation parameters.

While in this work we study the general class of non-convex optimization problems,
it can be interesting to investigate the global performance of our methods for convex objectives.
Indeed, it is well-known that, when the problem is convex, 
the rate of minimizing the gradient norm can be improved
and the methods can be accelerated \cite{grapiglia2019accelerated}.
Hence, it seems to be an important direction for future research to study the complexities
of first-order and zeroth-order regularized Newton schemes in convex case.

Another interesting question is related to comparison of our new schemes
with derivative-free implementation of the first-order and direct-search methods
\cite{nesterov2017random,bergou2020stochastic,gratton2015direct,grapiglia2023worst}.
These methods need at most $\mathcal{O}( n \epsilon^{-2} )$ function evaluations
to find a first-order $\epsilon$-stationary point 
(in expectation or with high probability for  stochastic methods \cite{nesterov2017random,bergou2020stochastic,gratton2015direct},
or in terms of the full gradient norm for a deterministic method \cite{grapiglia2023worst}).
We see that bound $\mathcal{O}( n \epsilon^{-2} )$ is worse than ours $\mathcal{O}( n^{3/2} \epsilon^{-3/2} )$
in terms of dependence on~$\epsilon$, but has a better dimension factor.
However, note that these complexity bounds are obtained for \textit{different problem classes},
assuming either the first or second derivative to be Lipschitz continuous.
Therefore, the development of universal schemes that can automatically achieve the best possible complexity bounds across various problem classes appears to be important, both from practical and theoretical perspectives.
We keep these questions for further research.

\bibliographystyle{plain}
\bibliography{bibliography}

\end{document}